\documentclass[fleqn]{article}
\usepackage{latexsym}
\usepackage{amsmath}
\usepackage{amssymb}
\usepackage{epsfig}
\usepackage{amsfonts}

\begin{document}
\newtheorem{definition}{Definition}[section]
\newtheorem{theorem}[definition]{Theorem}
\newtheorem{lemma}[definition]{Lemma}
\newtheorem{proposition}[definition]{Proposition}
\newtheorem{examples}[definition]{Examples}
\newtheorem{corollary}[definition]{Corollary}
\def\square{\Box}
\newtheorem{remark}[definition]{Remark}
\newtheorem{remarks}[definition]{Remarks}
\newtheorem{exercise}[definition]{Exercise}
\newtheorem{example}[definition]{Example}
\newtheorem{observation}[definition]{Observation}
\newtheorem{observations}[definition]{Observations}
\newtheorem{comments}[definition]{Comments}
\newtheorem{algorithm}[definition]{Algorithm}
\newtheorem{criterion}[definition]{Criterion}
\newtheorem{algcrit}[definition]{Algorithm and criterion}

\newenvironment{prf}[1]{\trivlist
\item[\hskip \labelsep{\it
#1.\hspace*{.3em}}]}{~\hspace{\fill}~$\square$\endtrivlist}
\newenvironment{proof}{\begin{prf}{Proof}}{\end{prf}}

\title{Solving linear differential equations}
\author{K.A.~Nguyen and M. van der Put \\
\footnotesize Department of Mathematics, University of Groningen,\\
\footnotesize P.O.Box 800, 9700 AV, Groningen, The Netherlands,\\
\footnotesize email  khuong@math.rug.nl and mvdput@math.rug.nl  }
\date{09-11-2007}

\maketitle\textrm{}

\begin{abstract}{The theme of this paper is to `solve' an absolutely irreducible differential module explicitly in terms of modules of lower dimension and finite extensions of the differential field  $K$. Representations of semi-simple Lie algebras and differential Galois theory are the main tools. The results  extend the classical work of G.~Fano. } \end{abstract}

\section*{Introduction}
L.~Fuchs posed the problem whether the $n$ independent solutions of a scalar linear
differential equation of order $n$ over $K=\mathbb{C}(z)$, under the assumption that these solutions
satisfy a non trivial homogeneous equation over $\mathbb{C}$, could be expressed in terms of solutions of scalar linear differential equations of lower order. G.~Fano wrote an extensive paper on this theme. His tools were an early form of the differential Galois group and an extensive knowledge of low dimensional projective varieties.

In the work of M.F.~Singer and in a recent paper of K.A.~Nguyen, a powerful combination of Tannakian methods and representations of semi-simple Lie algebras yields a complete solution to Fuchs' problem. We note that a  {\it scalar}
 differential equation is essential for Fuchs' problem. It is not clear whether this problem makes sense for a linear differential equation in the standard matrix form $Y'=AY$ (or in module form).

The theme of this paper is to explicitly  `solve' a differential equation of order $n$ in terms of  differential equations of lower order, whenever possible. This makes sense in terms of differential modules over $K$. It means that one tries to obtain a given differential module $M$ of dimension $n$, by constructions of linear algebra and, possibly, algebraic extensions of $K$, from differential modules of smaller dimension. In the sequel, $K$ will be a finite extension of $\mathbb{C}(z)$ (unless otherwise stated). For actual computer computations, one has to replace $\mathbb{C}$ be a `computable field', like the algebraic closure of $\mathbb{Q}$. We write $V=V(M)$ for the solution space of $M$ and $Gal(M)\subset {\rm GL}(V)$ for the differential Galois group of $M$. Further $(\frak{gal}(M),V)$ will denote the Lie algebra of
$Gal(M)$ acting on the solution space $V$.

If $M$ admits, for instance,  a non trivial submodule $N$, then $M$ is `solved' by $N,\ M/N$ and an element in $Ext^1(M/N,N)$ (corresponding to some
inhomogeneous equations $y'=f$ over the Picard-Vessiot field of $N\oplus M/N$). This is the reason that we will only consider irreducible modules $M$.

Fix an algebraic closure $\overline{K}$ of $K$. A module $M$
 over $K$ is called {\it absolutely irreducible} if
 $\overline{M}:=\overline{K}\otimes _KM$ is irreducible. Since
 $Gal(\overline{M})=Gal(M)^o$, this condition is equivalent to
 the statement that $V$ is an irreducible $Gal(M)^o$-module.

 If an irreducible module $M$ over $K$ becomes reducible after tensorization with $\overline{K}$, then this is a case where $M$ can be expressed, after a finite extension of $K$, into modules of
smaller dimension (see [C-W]). We will investigate this phenomenon in a future paper
and concentrate here (but not exclusively) on absolutely irreducible modules.

We will use the {\it notation}:
$\mu _n$ is the subgroup of order $n$ of $\mathbb{C}^*$, sometimes identified with scalar multiples
of the identity matrix or map; ${\bf 1}$ is the trivial module of dimension one and for a module $M$ of dimension $m$, we write $\det M =\Lambda ^mM$; further $M^*$ or $M^{-1}$ denotes the dual of $M$. The condition $\det M ={\bf 1}$ is equivalent to:
the matrix of $\partial$ with respect to a suitable basis of $M$ has trace 0.
For an absolutely irreducible module $M$ with $\det M={\bf 1}$ the group $Gal(M)^o$ is semi-simple and so is $\frak{gal}(M)$. Moreover $V(M)$ is an irreducible representation of $\frak{gal}(M)$. According to [S] and [N], $M$ {\it cannot} be solved in terms
of modules of lower dimensions and finite field extensions of $K$ if and only if $\frak{gal}(M)$ is simple and its representation $V(M)$ has smallest dimension among its non trivial representations. Let a scalar equation $L$ of order $n$ induce a representation of a simple Lie algebra of smallest dimension, then it is still possible that the $n$ independent solutions of
$L$ satisfy a non trivial homogeneous relation $F$ over $\mathbb{C}$ (in contrast to L.~Fuchs' opinion). The following list gives a complete answer. \\

\begin{small} {\it Simple Lie algebras, smallest dimension, degree of  $F$}\end{small}

\begin{tabular}{||c|c|c|c|c||}
\hline
symbol & & Lie algebra & smallest &$\deg F$\\
\hline
$A_n$ &  $n\geq 1$ & $\frak{sl}_{n+1}$ &$n+1$& NO\\

$B_n$ & $n\geq 3$ & $\frak{so}_{2n+1}$& $2n+1$& 2\\
$C_n$ & $n\geq 2$ & $\frak{sp}_{2n}$ &   $2n$& NO\\
$D_n$ & $n\geq 4$ & $\frak{so}_{2n}$ & $2n$& 2\\
$E_6$& &$\frak{e}_6$ & 27&3\\
$E_7$& &$\frak{e}_7$ & 56&4\\
$E_8$ & &$\frak{e}_8$ & 248&2 \\
$ F_4$& &$\frak{f}_4$&  26&2\\
$G_2$ & &$\frak{g}_2$ & 7&2\\
\hline
\end{tabular}
$\ $\\

We note that : $\frak{so}_3\cong \frak{sl}_2$, $\frak{so}_4\cong \frak{sl}_2\times \frak{sl}_2$,
$\frak{so}_5\cong \frak{sp}_4$ and
$\frak{so}_6\cong \frak{sl}_4$. \\

\noindent
The two cases where $N$ can be solved by modules of smaller dimension and finite field
extensions of $K$ are:
\noindent (a) $\frak{gal}(N)=\frak{g}_1\times \frak{g}_2$ and $V=V_1\otimes V_2$, where $V_i$ is an irreducible
 representation of $\frak{g}_i$ for $i=1,2$.
(b)  $\frak{g}:=\frak{gal}(N)$ is simple and the representation $V(N)$ does not have smallest dimension.\\
In  case (a) we produce an algorithm that expresses $N$ (after possibly a finite extension of $K$)
as a tensor product $N_1\otimes N_2$ with  $\dim N_i > 1$.\\
In  case (b) we produce a differential module $M$ corresponding to a representation of $\frak{g}$ of smallest
dimension and a construction of linear algebra by which $N$ is obtained from $M$. In the case that $Gal(N)$ is not
connected a finite (computable) extension of the field $K$ might be needed.

 Besides using the well known Tannakian methods (often called constructions of linear algebra)  a new construction (Theorem 1.3) is introduced which can be explained
as follows. An irreducible differential module $N$  is called {\it standard} if $Gal(N)$
is connected and the $\frak{gal}(N)$-module $V(N)$ is faithful of minimal dimension. Further an irreducible
differential module $M$ is called {\it adjoint} if $Gal(M)$ is connected and the $Gal(M)$-module $V(M)$ is
the adjoint representation. An adjoint differential module $M$ (for a given semi-simple Lie algebra) is obtained
as an irreducible submodule of ${\rm Hom}(N,N)$ where $N$ is a standard differential module (for the same semi-simple
Lie algebra). The new construction uses Lie algebra tools to go in the other direction, i.e., to obtain $N$ from $M$. Together with the Tannakian approach
the new construction provides a complete solution for both cases (a) and (b). For special
cases there are shortcuts not using adjoint differential modules.

For differential modules  of small dimension we rediscover and extend Fano's work. The finite group $Gal(M)/Gal(M)^o$ introduces a technical complication in the
method and is responsible for the finite extension of $K$  that are sometimes
needed. Our extensive use of  representations of semi-simple groups and semi-simple Lie algebras is a link between this paper and several chapters
of [K].

\section{Representations of semi-simple Lie algebras}

\subsection{General remarks on computations}

\noindent (1)
The fields $\mathbb{C}(z)$, $\overline{\mathbb{Q}}(z)$ and their finite extensions $K$ are $C_1$-fields. In particular a quadratic homogeneous form
over  $K$ in at least three variables has a non trivial zero and there are algorithms
(if $K$ is a `computable field')  producing such a zero, needed in some of the proposed computations.

For finite extensions of the second field there are efficient algorithms, due to M.~van Hoeij et al., implemented
in Maple, for finding submodules of a given `input module' $P$ (or, equivalently, for factoring differential operators over $K$). In particular, one can decide whether $P$ is irreducible. This algorithm is
less efficient for the verification that $P$ is absolutely irreducible. In the sequel we will suppose that the `input module' $P$ is absolutely irreducible. If $P$ happens to be irreducible but not absolutely irreducible, then the algorithms that we propose will either still work or demonstrate that $P$ is not
absolutely irreducible. \\

\noindent (2)
Suppose that the input module $P$ is (absolutely) irreducible. Then any
module $N$ obtained by a construction of linear algebra from $P$
is {\it semi-simple} (i.e., is a direct sum of irreducible submodules). A computation of $\ker (\partial ,{\rm End}(N))$ (i.e., the rational solutions of ${\rm End}(N)$) seems to be an efficient way to produce the direct summands of $N$. The characteristic polynomial of any $K$-linear $f:N\rightarrow N$ with $\partial f=0$ has coefficients in $\mathbb{C}$. Indeed, it coincides with the
characteristic polynomial of the induced map $V(f):V(N)\rightarrow V(N)$.
If $f$ is not a multiple of the identity on $N$, then any zero $\lambda \in \mathbb{C}$ of its characteristic polynomial yields a proper submodule
$\ker (f-\lambda ,N)$ of $N$.

If $N$ happens to be irreducible but is known to have direct
summands after a finite extension $K^+$ of the base field $K$, then one can explicitly compute $K^+$.

Consider, as an example (see also the example in Remarks 1.2 and the end of Section 3), the case where it is a priori known that $\overline{N}=A_1\oplus A_2$, where $A_1,A_2$ are non isomorphic irreducible differential modules over $\overline{K}$.
  Then
 $\ker (\partial ,{\rm End}(\overline{N}))=\mathbb{C}p_1+\mathbb{C}p_2$, where $p_1,p_2$ are the projections onto the two factors $A_1,A_2$ and
 $p_1+p_2=1$. The group $Gal(\overline{K}/K)$ acts as a finite group $H$
(faithfully) on this 2-dimensional vector space and the line $\mathbb{C}(p_1+p_2)$ is invariant. Therefore, there is another invariant line and $H$ is a cyclic group of order $d>1$. Thus ${\rm End}(N)$ contains precisely two 1-dimensional submodules, a trivial one and a non trivial one $L$ with
 $L^{\otimes d}={\bf 1}$. Then $L=Ke$ with $\partial e=\frac{g'}{dg}$ for some $g\in K^*$ and the field $K^+$ equals $K(\sqrt[d]{g})$.\\

\noindent (3)
(a) Suppose that two irreducible modules $M_1,M_2$ of the same dimension are given. An efficient way to investigate whether  $M_1$ and $M_2$ are isomorphic is to compute $\ker (\partial,
{\rm Hom}(M_1,M_2))$. This space is non zero if and only if  $M_1\cong M_2$.

\noindent
(b) Suppose that $M_2\cong M_1\otimes L$ holds for some unknown 1-dimensional module $L$. In order to find $L$ one considers the module $E:={\rm Hom}(M_1,M_2)={\rm Hom}(M_1,M_1)\otimes L$ and observes that
$L$ is a 1-dimensional direct summand of $E$. Using the method of (2) above
one can produce $L$.

\noindent
(c) Suppose that $M_1, M_2$ are absolutely irreducible and that
$\overline{M}_1\cong \overline{M}_2$. Then $\ker (\partial ,\overline{K}\otimes
{\rm Hom}(M_1,M_2))$ is a 1-dimensional vector space $\mathbb{C}\xi$.
Then $Gal(\overline{K}/K)$ acts as a cyclic group of order $d$ on
$\mathbb{C}\xi$. Thus  $\overline{K}\xi \subset  \overline{K}\otimes {\rm Hom}(M_1,M_2))$ is invariant under $Gal(\overline{K}/K)$ and induces a 1-dimensional submodule $L$ of ${\rm Hom}(M_1,M_2)$ with the properties
$L^{\otimes d}={\bf 1}$ and $M_2\cong M_1\otimes L$. Thus we can apply method (b) to find $L$. As in (2), $L$ defines a cyclic extension $K^+\supset K$ of degree $d$ and $K^+\otimes _KM_1\cong K^+\otimes _KM_2$.\\

 \noindent (4) Properties of a differential module $M$ over $K$ are often translated into
 properties of the (faithful) representation $(Gal(M),V(M))$ (and visa versa). By inverse Galois theory, any faithful representation of any linear algebraic group over
 $\mathbb{C}$ occurs for a differential field $K$ which is a finite extension of $\mathbb{C}(z)$.

\subsection{A table of irreducible representations}

We present here a list of irreducible representations $V,\ \dim V=d$,  of semi-simple Lie algebras, including the decomposition of $\Lambda ^2V$ and $sym^2V$.\\
{\it We adopt here and in the sequel of the paper
the efficient notation of the online program {\rm [LiE]} for irreducible representations}. \\
This is the following. After a choice of  simple roots $\alpha _1,\dots ,\alpha _d$, the Dynkin diagram
(with standard numbering of the vertices by the roots) and the fundamental weights $\omega _1,\dots ,\omega _d $ are well defined. The irreducible representation with
weight $n_1\omega _1+\cdots +n_d\omega _d$ is denoted by $[n_1,\dots ,n_d]$.
In particular, $[0,\dots ,0]$ is the trivial representation of dimension 1. \\

 \begin{center} {\it Table of the irreducible representations of dimension $d\leq 6$.}\end{center}

\begin{footnotesize}
\begin{tabular}{||l|l|l|l|l||}
\hline
d &Lie alg & repr & $\Lambda ^2$ & $sym ^2$  \\
\hline
$2$ &  $\frak{sl}_2$ & $[1]$ &$[0]$ &[2]\\
\hline
$3$ & $\frak{sl}_2$ & $[2]$& $[2]$ & $[4],[0]$ \\
\hline
$3$ & $\frak{sl}_3$ & $[1,0]$ &   $[0,1] $ & $[2,0]$\\
\hline
$4$&$\frak{sl}_2$ &$[3]$ &$[4],[0]$  & $[6],[2]$\\
\hline
$4$&$\frak{sl}_4$ &$[1,0,0]$ &$[0,1,0]$  &$[2,0,0]$\\
\hline
$ 4$&$\frak{sp}_4$ &$[1,0]$& $[0,1],[0,0]$ &$[2,0]$\\
\hline
$4$ &$\frak{sl}_2\times \frak{sl}_2$ &$[1]\otimes [1]$ &
$[0]\otimes [2] ,[2]\otimes [0]$  &$[0]\otimes [0],[2]\otimes [2]$ \\
\hline
$5$&$\frak{sl}_2$ & $[4]$ & $[6],[2]$ & $[8],[4],[0]$ \\
\hline
$5$ & $\frak{sp}_4$ & $[0,1]$ & $[2,0]$ & $[0,2],[0,0]$ \\
\hline
$5$ & $\frak{sl}_5$ & $[1,0,0,0]$ & $[0,1,0,0]$ & $[2,0,0,0]$ \\
\hline
$6$ & $\frak{sl}_2$ & $[5]$ & $[8],[4],[0]$ & $[10],[6],[2]$\\
\hline
$6$ & $\frak{sl}_3$ & $[2,0]$ & $[2,1]$ & $[4,0],[0,2]$\\
\hline
$6$&$\frak{sl}_4$&$[0,1,0]$&$[1,0,1]$&$[0,2,0],[0,0,0]$\\
\hline
$6$ & $\frak{sl}_6$ & $[1,0,0,0,0]$ & $[0,1,0,0,0]$ & $[2,0,0,0,0]$\\
\hline
$6$ &$\frak{sp}_6$&$[1,0,0]$&$[0,1,0],[0,0,0]$&$[2,0,0]$\\
\hline
$6$&$\frak{sl}_2\times \frak{sl}_2$&$[1]\otimes [2]$&
$[0]\otimes [0],[0]\otimes [4],[2]\otimes [2]$&
$[0]\otimes [2], [2]\otimes [0],[2]\otimes [4]$\\
\hline
$6$&$\frak{sl}_2\times \frak{sl}_3$&$[1]\otimes [1,0]$&
$[0]\otimes [2,0],[2]\otimes [0,1]$&$[0]\otimes [0,1],[2]\otimes [2,0]$ \\
\hline
\end{tabular}
\end{footnotesize}

$\ $\\

\noindent
For the $\frak{sl}_n$ with $n>2$ we have omitted duals of representations. Further we have left out symmetric cases.
 The decompositions of the second symmetric power
and the second exterior power are useful to distinguish the various
cases. We are here especially interested in those representations which can be expressed in terms of representations of lower dimension. In dimensions $7-11$, one finds for the new items of this sort
(here we omit the case $\frak{sl}_2$ which is fully treated in section 2
 and again we omit duals and symmetric situations) the list:\\
\noindent $\frak{sl}_3$ with $[1,1]$ (dim 8), $[3,0]$ (dim 10); $\ \ $ $\frak{sl}_4$ with $[2,0,0]$ (dim 10);\\
$\frak{sl}_5$ with $[0,1,0,0]$ (dim 10);$\ $ $\frak{so}_7$ with $[0,0,1]$ (dim 8); $\ $
$\frak{sp}_4$ with $[2,0]$ (dim 10);\\
\noindent $\frak{sl}_2\times \frak{sl}_2$ with $[1]\otimes [3]$ (dim 8), with $[2]\otimes [2]$ (dim 9), with $[1]\otimes [4]$ (dim 10);\\
$\frak{sl}_2\times \frak{sl}_3$ with $[2]\otimes [1,0]$ (dim 9); $\ \ $$\frak{sl}_2\times \frak{sl}_4$ with $[1]\otimes [1,0,0]$ (dim 8);\\
$\frak{sl}_2\times \frak{sp}_4$ with $[1]\otimes [1,0]$ (dim 8), with $[1]\otimes [0,1]$ (dim 10);\\
 $\frak{sl}_2\times \frak{sl}_5$ with $[1]\otimes [1,0,0,0]$ (dim 10); $\ \ $
$\frak{sl}_3\times \frak{sl}_3$ with $[1,0]\otimes [1,0]$ (dim 9);\\
$\frak{sl}_2\times \frak{sl}_2\times \frak{sl}_2$ with $[1]\otimes [1]\otimes [1]$ (dim 8).\\

\subsection{Comparison of the representations of $G$ and  $\frak{g}$}

For a {\it connected} semi-simple group $G$ with Lie algebra $\frak{g}$ one considers the categories $Repr_G$ of the representations of $G$ on finite dimensional vector spaces
over $\mathbb{C}$ and $Repr_{\frak{g}}$, the category of the representations of $\frak{g}$
on finite dimensional vector spaces over $\mathbb{C}$. Any representation of $G$ on a vector space induces a representation of $\frak{g}$ on the same vector space.  This defines a functor $T:\ Repr_G\rightarrow Repr_{\frak{g}}$, which is fully faithful,  i.e.,
 ${\rm Hom}_G(V_1,V_2)\rightarrow  {\rm Hom}_{\frak{g}}(V_1,V_2)$ is a bijection.  Further, $G$ is simply connected if and only if $T$ is an equivalence.

For a representation $W\in Repr_G$, we write $\{\{W\}\}$ for the Tannakian subcategory generated by $W$ (i.e., the objects of this subcategory are obtained from $W$ by constructions of linear algebra). The action of $G$ is faithful if and only if $\{\{W\}\}=Repr_G$. Similarly, for an object $W\in Repr_{\frak{g}}$ one writes $\{\{W\}\}$ for the smallest Tannakian subcategory generated by $W$.

 Suppose that $\frak{g}$ acts faithfully on $W$, then in general $\{\{W\}\}\neq Repr_{\frak{g}}$. Indeed, let $G^+$ be the simply connected group
with Lie algebra $\frak{g}$. Then $W$ has a unique structure as $G^+$-module compatible with its structure as $\frak{g}$-module. The kernel $Z'$ of the action of $G^+$
on $W$ is a finite group. Put $H:=G^+/Z'$. Then $W$ is a faithful $H$-module and generates $Repr_H$. Thus the subcategory $\{\{W\}\}$ of $Repr_{\frak{g}}$ is the image
under $T$ of $Repr_H$.\\

\noindent
{\it  Example}: $G={\rm SL}_3$ is simply connected and $\frak{g}=\frak{sl}_3$. There is only one other connected group with Lie algebra $\frak{sl}_3$, namely
${\rm PSL}_3={\rm SL}_3/\mu _3$. Let $V$ be the standard representation of
${\rm SL}_3$ with $T$-image $(\frak{sl}_3,[1,0])$. Then $sym^3V$ is a faithful representation for ${\rm PSL}_3$ and its image under $T$ is  $W:=(\frak{sl}_3,[3,0])$.
Then $\{\{V\}\}=Repr _{\frak{sl}_3}$ and $\{\{W\}\}$ is the full subcategory of $Repr_{\frak{sl}_3}$ for  which the irreducible objects are the $[a,b]$ with $a\equiv b\mod 3$.\\
\noindent  {\it Consequences for differential modules}. Let the input module $P$ be an absolutely irreducible differential  module with $\det P={\bf 1}$ which induces
$W:=(\frak{sl}_3,[a,b])$ with, say, $[a,b]\neq  [1,0],[0,1]$.  Now, we do not assume that $Gal(P)$ is connected. There exists, as we know, a differential module $M$ of dimension 3 inducing $(\frak{sl}_3,[1,0])$ such that $P$ is obtained from $M$ by constructions of linear algebra and possibly a finite field extension of $K$.

If $a\not \equiv b\mod 3$, then $[1,0]$ is obtained from $W$ by a construction $cst_1$ of linear algebra and $[a,b]$ is (of course) obtained by a construction $cst_2$ from
$[1,0]$. Let $M$ be obtained from $P$ by construction $cst_1$. Then $cst_2$ applied
to $M$ yields a module $\tilde{P}$ which is isomorphic to $P$ over the algebraic closure
of $K$. Case (3)(c) of  Subsection 1.1 solves this.

If $a\equiv b\mod 3$, then one obtains by a construction of linear algebra a module $Q$
which induces $(\frak{sl}_3,[1,1])$.  The step from $Q$ to a module which induces
$(\frak{sl}_3,[1,0])$ cannot be done by constructions of linear algebra. This  {\it problem}
is an example for the main theme of the remainder of this section.\\

In general, the problem has its origin in the possibilities for the connected groups with a given (semi-) simple Lie algebra $\frak{g}$. There is a simply connected group $G$ with Lie algebra $\frak{g}$. Its center $Z$ is a finite group. Any connected group with Lie algebra $\frak{g}$ has the form $G/Z'$ where $Z'$ is a subgroup of $Z$. The list of the groups $Z$ that occur  is, see [H], p. 231:
$\mathbb{Z}/(n+1)\mathbb{Z}$ for $A_n$; $\ \ \mathbb{Z}/2\mathbb{Z}$ for $B_\ell,C_\ell,E_7$;
$\ \ \mathbb{Z}/2\mathbb{Z}\times \mathbb{Z}/2\mathbb{Z}$ for $D_\ell$ with $\ell$ even;
$\mathbb{Z}/4\mathbb{Z}$ for $D_\ell$ with $\ell$ odd; $\mathbb{Z}/3\mathbb{Z}$ for $E_6$; $0$ for $E_8,F_4,G_2$.
The following proposition, closely related to Corollary 2.2.2.1 of [K], is a non constructive solution to our problem.

\begin{proposition} Let $G^+\rightarrow G$ be a surjective morphism of connected linear
algebraic groups over $\mathbb{C}$ having a finite kernel $Z$. Let $M$ be a differential module over $K$ with $Gal(M)=G$. Suppose that
$K$ is a $C_1$-field. Then there  exists a differential module $N$ over $K$ with $Gal(N)=G^+$, such that the faithful representation $(Gal(N),V(N))$ has minimal dimension and  such that  $M\in \{\{N\}\}$.
\end{proposition}
\begin{proof} The Picard-Vessiot ring $PVR$ of $M$ is a $G$-torsor over $K$. Since $K$ is a $C_1$-field and $G$ is connected, this torsor is trivial and thus
$PVR=K\otimes _\mathbb{C}\mathbb{C}[G]$, where $\mathbb{C}[G]$ is the coordinate
ring of $G$. The $G$-action on $PVR$ is induced by the $G$-action on $\mathbb{C}[G]$.
The differentiation, denoted by $\partial$, commutes with the $G$-action, but is not
explicit. Now $\mathbb{C}[G]=\mathbb{C}[G^+]^Z$ and this yields an embedding
$PVR\subset R:=K\otimes _\mathbb{C}\mathbb{C}[G^+]$. The differentiation $\partial$ on $PVR$
extends in a unique way to $R$ since $PVR\subset R$ is a finite \'etale extension. The extended differentiation commutes with the action of $G^+$.
Further, $R$ has only trivial differential ideals since $R$ is finite over $PVR$ and $PVR$ has only trivial differential ideals. Let $W$ be a
faithful representation of $G^+$ of minimal dimension $d$. Then one writes
$\mathbb{C}[G^+]=\mathbb{C}[\{X_{i,j}\}_{i,j=1,\dots d},\frac{1}{D}]/J$, where
$D=\det (X_{i,j})$ and $J$ is the ideal defining $G^+$ as subgroup of ${\rm GL}(W)$.
Write $x_{i,j}$ for the image of $X_{i,j}$ in $\mathbb{C}[G^+]$. Define the matrix
$A$, with entries in $R$, by $(\partial x_{i,j})=A(x_{i,j})$. Then $A$ is invariant under the action of $G^+$
and therefore its entries are in $K$. It now follows that $R$ is the Picard-Vessiot ring
for the differential equation $Y'=AY$. This equation defines the required differential
module $N$ over $K$.\end{proof}

\begin{remarks} Non connected linear algebraic groups. {\rm \\
Let $G$ be a linear algebraic group such that $G^o$ is semi-simple
and $G\neq G^o$. Let $\frak{g}$ denote the Lie algebra of
$G^o$. As explained above, the functor
$T: Repr_{G^o}\rightarrow Repr_{\frak{g}}$ induces an equivalence of
the first category with a well described full subcategory of the second one. The forgetful functor $F: Repr_G\rightarrow Repr_{G^o}$ is not fully faithful. Indeed, one can easily construct an irreducible $G$-module  $W$ (i.e., a finite dimensional representation of $G$) which is, as $G^o$-module, the direct sum of several copies of an irreducible $G^o$-module.

 A more delicate situation occurs when $Out(G^o):=Aut(G^o)/Inner (G^o)$ is not trivial. The group $Out(G^o)$ permutes the irreducible
 representations of $G^o$. The action by conjugation of $G$ on
 $G^o$ induces a homomorphism $G/G^o\rightarrow Out(G^o)$.
 If this homomorphism is not trivial, then one can construct an    irreducible $G$-module $W$ such that $W$, seen as a $G^o$-module, is a direct sum of distinct irreducible $G^o$-modules forming
a single orbit under the action of $G/G^o$.

We recall that for a connected simple $H$ the group $Out$ is equal to
the automorphism group of the Dynkin diagram of its Lie algebra $\frak{h}$.
According to [J], Theorem 4, p. 281, one has:\\
 $Out=S_3$ for $\frak{so}_8$;
 $Out=\mathbb{Z}/2\mathbb{Z}$ for $\frak{sl}_n,\ n>2$, for  $\frak{so}_{2n},
\ n\geq 3, n\neq 4$ and for $\frak{e}_6$.
 For the other {\it simple} Lie algebras $Out$ is  trivial.
>From this list one deduces $Out$ for any semi-simple Lie algebra, e.g.,
$Out(\frak{sl}_2\times \frak{sl}_2\times \frak{sl}_2)=S_3$.\\

\noindent {\it Example}. $Out({\rm SL}_4)$ is generated by the element
$A\mapsto (A^t)^{-1}$. This non trivial element changes a representation $[a,b,c]$ of $\frak{sl}_4$
 (or of ${\rm SL}_4$) into its dual $[c,b,a]$.
Choose a group $G$ with $G^o={\rm SL}_4$,
 $[G: G^o]=2$ and $G/G^o\rightarrow Out(G^o)$ is bijective.
Then there is an irreducible $G$-module $W$ which induces
the $\frak{sl}_4$-module $[1,0,0]\oplus [0,0,1]$. Thus $W$ is reducible
as $G^o$-module.

 \noindent {\it Consequence for differential modules}.  Let $M$ be an absolutely irreducible differential module with
 $\det M={\bf 1}$ such that the
 $\frak{gal}(M)$-module $V(M)$ has a non trivial direct sum decomposition. Then, in general, a finite field
 extension of $K$  is needed to obtain a corresponding direct sum decomposition of $M$.  }\end{remarks}

\subsection{Standard and adjoint differential modules}

Let $G$ be simply connected semi-simple linear algebraic group over $\mathbb{C}$ with Lie algebra
$\frak{g}$. One writes $G=G_1\times \cdots \times G_s$ and $\frak{g}=\frak{g}_1\times \cdots \times
\frak{g}_s$ for the decompositions into simple objects. The {\it standard representation} of $G$
is the direct sum $V=\oplus _{i=1}^sV_i$, where each $V_i$ is the faithful representation of $G_i$ of
 smallest dimension. For the groups
${\rm SL}_n,\ n>2$, both the standard representation and its dual have
smallest dimension. This ambiguity in the above definition is of no importance for the sequel.
 The $G$-module ${\rm End}(V)=V^*\otimes V$ has $\frak{g}$ as irreducible submodule. The action
 of $G$ on $\frak{g}$ is the {\it adjoint
representation}. The kernel of this action is the finite center $Z$ of $G$ and
$G/Z$ is the adjoint group.

A differential module $M$ over $K$ is called {\it standard} ({\it adjoint} resp.) for $G$  if $Gal(M)$
 is connected, $\det M={\bf 1}$ and $(Gal(M),V(M))$ is isomorphic to the standard (adjoint resp.)
 representation of $G$. For any standard differential module $M$, the module ${\rm End}(M)$ contains
a unique direct summand which is an adjoint differential module.

\bigskip

Let $G$  be a simply connected semi-simple group with Lie algebra $\frak{g}$
and let $V$ be the standard representation of $G$. An {\it explicit standard module for $G$}
is the following. Write $M=K\otimes _{\mathbb{C}}V$.
Then $\frak{g}$ is identified, as before, with a subspace of
${\rm End}(V)$ and $\frak{g}(K)=K\otimes \frak{g}$ is identified with the $K$-vector space
 $K\otimes \frak{g}\subset {\rm End}_K(M)$.  Define the derivation
$\partial _0$ on $M$ by $\partial _0$ is zero on $V$. For any element  $S$ of $\frak{g}(K)$ one
 defines the derivation $\partial _S:=\partial _0+S$ on $M$. According to Proposition 1.31 of [vdP-S],
 the differential Galois group
of $(M,\partial _S)$ is contained in $G$. We call $M=(M,\partial_S)$ an explicit standard module for
 $G$ if the differential Galois group is equal to $G$. According to Section 1.7 of [vdP-S],
 the differential Galois group is equal to $G$ for sufficiently general $S$. On the other hand,
 under the assumption that
$K$ is a $C_1$-field, every standard differential module for $G$ has an explicit representation
 $(M,\partial _S)$ by Corollary 1.32 of [vdP-S].

 For any explicit standard module $(M,\partial _S)$, one considers the direct summand
 $N:=K\otimes _{\mathbb{C}}\frak{g}$ of ${\rm End}_K(M)$.  Define the derivation $\partial _0$ on
 $N$ by $\partial _0$ is zero on $\frak{g}$. One easily verifies that $(M,\partial _S)$ induces on $N$ the derivation
 $A\mapsto \partial _0(A)+[A,S]$. In this way $(M,\partial _S)$ induces an
 adjoint differential module.

 \begin{theorem} Let $N$ be an adjoint differential module for $G$. The
 $\mathbb{C}$-Lie algebra structure of $\frak{g}=V(N)$ induces a $K$-Lie algebra structure
 $[\ ,\ ]$ on $N$ satisfying $\partial [a,b] =[\partial a,b]+[a,\partial b]$ for all  $a,b\in N$.
 This structure is unique up to multiplication by an element in $\mathbb{C}^*$.

  The assumption that $K$ is a $C_1$-field implies that there exists an isomorphism of $K$-Lie
 algebras  $\phi :N\rightarrow K\otimes _{\mathbb{C}}\frak{g}$. After choosing $\phi$, there
 exists a unique $S\in \frak{g}(K)$ such that $N$ is isomorphic to the adjoint module induced by
 the explicit standard module $(M,\partial _S)$.
 \end{theorem}
 \begin{proof}  By definition, $(Gal(N),V(N))$ is the adjoint action of $G$ (or $G/Z$) on
 $\frak{g}$. The morphism of $G$-modules $\Lambda ^2\frak{g}\rightarrow \frak{g}$, given by
 $A\wedge B\mapsto [A,B]$, is $G$-equivariant and therefore is induced by a morphism of differential
 modules $F: \Lambda ^2N\rightarrow N$.  The map $F$ is a non zero element
 of $\ker (\partial ,{\rm Hom}(\Lambda ^2N,N))$, unique up to multiplication by an element in $\mathbb{C}^*$.
 Define for $a,b\in N$ the
 expression $[a,b]=F(a\wedge b)$. The statement $\partial F=0$
 is equivalent to $\partial [a,b]=[\partial a,b]+[a,\partial b]$ for all $a,b\in N$.

 \smallskip

 Let $PVR(N)$ denote the Picard-Vessiot ring for $N$. The canonical isomorphism
 $PVR(N)\otimes _{\mathbb{C}}V(N)\rightarrow PVR(N)\otimes _KN$ is, by construction, compatible
 with the Lie algebra structures on $N$ and
 $V(N)=\frak{g}$. Since $Gal(N)$ is connected and $K$ is a $C_1$-field, the $Gal(N)$-torsor $Spec(PVR(N))$ over $K$ is
 trivial (see [vdP-S]). This means that there is a
 $K$-algebra homomorphism $e: PVR(N)\rightarrow K$. One applies $e$
 to both sides of the above canonical isomorphism and finds an isomorphism of
 $K$-Lie algebras  $\phi : K\otimes _{\mathbb{C}}\frak{g}\rightarrow N$.

 \smallskip

  After choosing $\phi$ and identifying $N$ with
 $K\otimes _{\mathbb{C}}\frak{g}$,  one can define the derivation $\partial _0$
on $N$ by $\partial _0$ is zero on $\frak{g}$. Clearly, $\partial _0[a,b]=
[\partial _0 a,b]+[a,\partial _0 b]$ for all $a,b\in N$.
Therefore $\partial -\partial _0$
is a $K$-linear derivation of the semi-simple Lie algebra $N$ and there is a unique
 $S\in N=\frak{g}(K)$ such that $\partial -\partial _0=[\ \ ,S]$ (see [J], theorem 9, p.80).
 Thus we find that $\partial$ on $N$ is induced by $(M,\partial _S)$.  \end{proof}

\begin{comments}. {\rm  The computation of the Lie algebra structure $F$ on $N$
amounts to computing a rational solution (i.e., with coordinates in $K$) of the differential
 module ${\rm Hom}(\Lambda ^2N,N)$. The computation of $S\in \frak{g}(K)$ is an easy exercise on
 Lie algebras. The  computation of an isomorphism $\phi$ seems more complicated. It is essentially
 equivalent to the computation
of a Cartan subalgebra of $N$ which is `defined' over $\mathbb{C}$. As an example we  consider here
 the case $\frak{g}=\frak{sl}_2$.

\smallskip

Since $N\cong K\otimes _{\mathbb{C}}\frak{sl}_2$, we know that there exists
an element $h\in N$ such that the eigenvalues of $ad (h)$, acting upon $N$, are $0,\pm 2$.
 Such an element $h$ can be found by solving some quadratic equation over $K$. Any other
 candidate $h'$ is conjugated to $h$ by an automorphism of the $K$-Lie algebra $N$. We
 choose an element $e_1$ with
$[h,e_1]=2e_1$ and an element $e_2$ with $[h,e_2]=-2e_2$. The last element
is multiplied by an element in $K^*$ such that moreover $[e_1,e_2]=h$ holds.
The $\mathbb{C}$-subspace of $N$ generated by $h,e_1,e_2$ is isomorphic to
$\frak{sl}_2$ and we have found an isomorphism $\phi$.

The element  $h$, which generates a Cartan subalgebra for
$N$, defined over $\mathbb{C}$, is essentially unique. The vectors $e_1,e_2$
 can however be replaced by $fe_1,f^{-1}e_2$ for any $f\in K^*$. This reflects
  the observation that the differential module $M$, with $\det M={\bf 1}$ and differential
 Galois group ${\rm SL}_2$, that induces the adjoint module $N$ is not unique. In fact, one
 can replace $M$ by $Ke\otimes _KM$ where the 1-dimensional module $Ke$ is given  by
 $\partial e=\frac{f'}{2f}e$ with $f\in K^*$.  }\end{comments}

\subsection{A general solution to the problem}
We recall the following.
${\rm Diff}_K$ will denote the (neutral) Tannakian category of all differential modules over $K$. The Tannakian group of this category is
an affine group scheme $\mathcal{U}$ (this is the universal differential Galois group), i.e.,  we have an equivalence ${\rm Diff}_K\rightarrow Repr_{\mathcal{U}}$ of Tannakian categories. For any differential module $P$ over $K$, we denote by $\{\{P\}\}$ the
full Tannakian subcategory generated by $P$. The module $P$ corresponds a representation $\rho :\mathcal{U}\rightarrow {\rm GL}(V(P))$. Its image is $Gal(P)$.
By differential Galois theory, there is a natural equivalence of Tannakian categories $\{\{P\}\}\rightarrow Repr_{Gal(P)}$.

We note that the constructions of linear algebra in the category  $Repr_{\frak{g}}$ for a semi-simple $\frak{g}$ are known and can be found explicitly  by, for instance, the online program [LiE].\\

\noindent
{\bf  The problem}. $P$ is an input differential module. Find a differential module of smallest dimension $M$ such that $P\in \{\{M\}\}$ or, more precisely, some differential modules $M_1,\dots ,M_r$ with $\max \{ \dim M_i\}$ as small as possible such that $P$ lies in the Tannakian subcategory  $\{\{M_1,\dots ,M_r\}\}$  generated by  $\{M_1,\dots ,M_r\}$.

\smallskip
\noindent
{\bf A solution to the problem}. Suppose that the input module $P$ has the properties
{\it absolutely irreducible, $\det P={\bf 1}$ and $Gal(P)$ is connected}. Then $\frak{g}:=\frak{gal}(P)$ is semi-simple. Let $G$ be, as before, the simply connected group with Lie algebra $\frak{g}$.  Then  $Gal(P)=G/Z'$ for some subgroup $Z'$ of the center $Z$ of $G$.
The adjoint representation $(G/Z,\frak{g})$ lies in $Repr_{G/Z'}$. The construction of linear algebra $csrt(1)$ from $(G/Z',V(P))$ to $(G/Z,\frak{g})$ can be read off in the equivalent subcategory $\{\{(\frak{g},V(P) )\}\}$ of $Repr_{\frak{g}}$.
 Using the equivalence of
$\{\{P\}\}$ and $Repr_{G/Z'}$ one can apply  $csrt(1)$ to $P$. This yields $N\in \{\{P\}\}$ which maps to the adjoint representation $(G/Z,\frak{g})$. Thus $N$ is an adjoint differential module for $G$. Theorem 1.3 provides a standard module $M$ for $G$ which induces $N$.

 Using  the equivalences
 $\{\{M\}\}\rightarrow Repr_G\rightarrow Repr_{\frak{g}}$  one finds an
explicit construction of linear algebra $csrt(2)$ from $M$ to a differential module $Q$ with $(Gal(Q),V(Q))=(Gal(P),V(P))$. Now $P$ and $Q$ are
almost isomorphic.

What we know is that $csrt(1)$ applied to $P$ and
$Q$ produce $N$. Let $\rho: \mathcal{U}\rightarrow {\rm SL}(V(P))$ and $\rho ':\mathcal{U}\rightarrow {\rm SL}(V(Q))$ denote the representation corresponding to $P$ and $Q$.
The fact that $\rho$ and $\rho '$ yield the same representation $\rho '':\mathcal{U}\rightarrow
{\rm SL}(\frak{g})$, corresponding to $N$, implies that $\rho$ and $\rho '$ are projectively equivalent, i.e., $\rho (u)=c(u)\rho '(u)$ for all $u\in \mathcal{U}$ and with $c(u)\in \mathbb{C}^*$ (in fact $c(u)^n=1$ where $n=\dim _KP$). Let the 1-dimensional differential module $L$ correspond to the representation $\mathcal{U}\rightarrow \mathbb{C}^*,\ u\mapsto c(u)$. Then  $P\cong L\otimes _K Q$.  Finally, $L$ can be made explicit by Subsection 1.1. part
(3). Thus we have explicitly found $P\in \{\{M,L\}\}$. In the case
 that $G$ is semi-simple but not simple we write, as before,
 $G=G_1\times \cdots \times G_s$. The standard module $M$ is a direct
 sum $M_1\oplus \cdots \oplus M_s$ where each $M_i$ is a standard
 module for the simple $G_i$. Thus we have
 $P\in \{\{M_1,\dots ,M_s,L\}\}$ and this is a solution to our problem.

\begin{comments}
{\rm
(1) {\it Variations}.  In many cases there are {\it shortcuts}. \\
(1.1) If $Z'=\{1\}$, then $\{\{P\}\}\cong Repr_G$ and
there is an explicit construction  $csrt(3)$ from
$(\frak{g},V(P))$ to the standard module  $(\frak{g},V)$. Then
$csrt(3)$ applied to $P$ yields the standard differential module
$M$ with $P\in \{\{M\}\}$.\\
(1.2) If $Z'=Z$, then there is a construction of linear algebra from the
adjoint module $N$ (obtained from $P$) back to $P$.\\
(1.3) If the adjoint $G$-module $\frak{g}$ is not the faithful
$G/Z$-module of smallest dimension, then we may use a faithful $G/Z$-module
of smallest dimension at the place of $\frak{g}$. Example:
For $G=Sp_4$ the group $Z=\{\pm 1\}$ and $Sp_4/Z=SO_5$. The natural representation of $SO_5$ has dimension 5 and $\frak{sp}_4$ has dimension 10.\\

\noindent
(2) {\it Non connected groups}. More generally, we may consider absolutely irreducible differential modules $P$ with $\det P={\bf 1}$. This assumption is equivalent to the statement that  $V(P)$ is an irreducible $Gal(P)^o$-module. The finite group
$Gal(P)/Gal(P)^o$ introduces (in general) two kinds of obstructions  to the above method. Consider namely a $Gal(P)$-module $W$, obtained by some construction of linear algebra from $P$ and $V(P)$. The irreducible summands $\{W_i\}$ of $(Gal(P)^o,W)$ can be permuted by $Gal(P)$ because the image of $Gal(P)/Gal(P)^o$ in $Out(\frak{gal}(P))$  is not trivial. The other possible obstruction can occur when  some $W_i$ has multiplicity greater than one. A computable finite extension $\tilde{K}$ of the base field $K$  is needed to make this subspace invariant under the new (smaller) differential Galois group of $\tilde{K}\otimes _KP$ over $\tilde{K}$.

We study this in more detail for the case that $(Gal(P)^o,V(P))$ {\it is the adjoint representation}.
 After identifying $V(P)$ with $\frak{g}$, the group $Gal(P)$ is contained in the group $G^+=G^{++}\cap {\rm SL}(\frak{g})$, where $G^{++}$ is the normalizer of $G^o:=Gal(P)^o$ in ${\rm GL}(\frak{g})$.  An element $T\in {\rm GL}(\frak{g})$
 belongs to $G^{++}$ if and only if there exists a constant $c\in \mathbb{C}^*$ such that $[TA,TB]=c[A,B]$ for all $A,B\in \frak{g}$.  One obtains an exact sequence
$1\rightarrow (\mu _n\cdot G^o)/G^o \rightarrow  G^+/G^o\rightarrow Out(G^o)\rightarrow 1$, where $n$ is the dimension of $P$.

 If the image of $Gal(P)$ in $Out(G^o)$ is not trivial, then one has to compute a finite field extension of $K$ which kills this part of $Gal(P)$.  If the image is trivial, then one can replace $P$ by the direct summand $\tilde{P}$ of $P\otimes P^*$ which is an adjoint representation and one obtains a standard module $\tilde{M}$ with
 $\tilde{P}\in \{\{\tilde{M}\}\}$.  Further $P\cong \tilde{P}\otimes L$, where $L$ is a 1-dimensional differential module such that $L^{\otimes n}={\bf 1}$. As mentioned in Subsection 1.1, there is an easy algorithm for the computation of $L$.     \\

 \noindent (3) In the Sections 3--6 we investigate special cases of shortcuts and non connected groups.  This includes all cases, listed in Subsection 1.2,  where a differential module can be `solved' in terms of modules of lower dimension and field extensions.   }\end{comments}

\section{Symmetric powers of modules of dimension 2}

In this section we make the method of Section 1 explicit for $\frak{sl}_2$. There are two connected algebraic groups with Lie algebra $\frak{sl}_2$, namely ${\rm SL}_2$ and ${\rm PSL}_2$. The first group corresponds to $Repr_{\frak{sl}_2}$ and the second group to the full subcategory of
$Repr_{\frak{sl}_2}$  for which the irreducible objects are $\{[2n]\ |\ n\geq 0\}$.

By operations of linear algebra one can obtain from
the object $[2n]$ with $n>1$, the object $[2]$. Indeed, one has
$\Lambda ^2[2n]=\oplus _{k=1}^{n}[4k-2]$.

Similarly, consider the object $[2n+1]$ with $n>0$. Then
$sym^2[2n+1]=\oplus _{k=1}^{n+1}[4k-2]$ and this yields $[2]$. Further $[2]\otimes [2k+1]=[2k-1]\oplus [2k+1]\oplus [2k+3]$ and this yields $[2k-1]$. In this way $[1]$ is obtained by operations of linear algebra from $[2n+1]$.  The formulas, used here, follow easily from the characters formulas for tensor products. \\

Let $M$ be an absolutely irreducible differential module with $\det M={\bf 1}$ and $(\frak{gal}(M),V(M))=(\frak{sl}_2,[n])$. The above construction from $[n]$ to $[1]$ or $[2]$ can be copied for the differential module $M$, since the dimensions of the direct summands are distinct and thus these direct summands
are not only invariant under $Gal(M)^o$ but also under $Gal(M)$. Moreover, one can verify that
$Gal(M)$ is contained in $\mu _{n+1}\cdot sym^n({\rm SL}_2)$ and this has the same consequence.\\
The essential problem to solve is:
{\it  Let $M$ be an absolutely irreducible module with $\det M={\bf 1}$ and $(\frak{gal}(M),V(M))=(\frak{sl}_2,[2])$. Produce a module $N$ of dimension 2 with $\det N={\bf 1}$ and  $sym^2N=M$}.\\

We start by a result due,  in various forms, to G.~Fano, M.F.~Singer and  M.~van Hoeij -  M.~van~der ~Put. The result is based on the following {\it observation}. Let $N$ be a differential module
with basis $n_1,n_2$ and $\det N={\bf 1}$.  Then $M:=sym ^2N$ has basis $m_{11}=n_1\otimes n_1,\ m_{22}=n_2\otimes n_2,\ m_{12}=n_1\otimes n_2$. The element
$F=m_{12}\otimes m_{12}-m_{11}\otimes m_{22}\in sym ^2M$ satisfies $\partial F=0$ and
$F$ is a non degenerate and has a non trivial isotropic vector.

\begin{theorem} Let $M$ be a differential module of dimension 3 with $\det M={\bf 1}$.  Suppose that there exists $F\in sym^2M$ such that $F$ is non degenerate and has a non trivial isotropic vector (no further conditions on $M$ or $K$). Then there exists a differential module $N$ of dimension 2 with $\det N={\bf 1}$ and $sym ^2N\cong M$.
\end{theorem}
\begin{proof}  $F$ has the form $m_{12}\otimes m_{12}-m_{11}\otimes
m_{22}$ for a suitable basis $\{m_{11},m_{22},m_{12}\}$ of $M$.
Consider a $K$-vector space $N$ with basis $n_1,n_2$ and define an isomorphism of
$K$-vector spaces $\phi :sym ^2N\rightarrow M$ by sending $n_1\otimes n_1,\ n_2\otimes n_2,\ n_1\otimes n_2$ to the elements $m_{11},m_{22},m_{12}$. We want to provide $N$
with a structure of differential module by putting $\partial n_i=\sum \alpha _{j,i}n_j$ for a
some matrix $(\alpha _{i,j})$ with trace 0, such that $\phi$ becomes an isomorphism
of differential modules.  Let the matrix of $\partial $ on $M$ be given by
$\partial m_{ij}=\sum \gamma _{kl,ij}m_{kl}$. The assumption $\partial F=0$ leads to the relations
\[ \gamma _{11,22}=\gamma _{22,11}=\gamma _{12,12}=\gamma _{11,11}+
\gamma _{22,22}=0, \ \gamma _ {11,12}=\gamma _{12,22}/2,\ \gamma _{22,12}=\gamma _{12,11}/2\ .\]
The condition that $\phi$ commutes with $\partial$ leads to the unique solution
\[\alpha _{1,2}= \gamma _{11,12},\ \alpha _{2,1}=\gamma _{22,12} ,\ \alpha _{1,1}=\gamma _{11,22}/2 ,\ \alpha _{2,2}= \gamma _{22,22}/2 \ . \] \end{proof}

\begin{corollary}[Test] (Let $K$ be a $C_1$-field). $M$ is an irreducible module, $\dim M=3,\ \det M={\bf 1}$. Then $M$ is isomorphic to a $sym^2N$ with $\dim N=2$,  $\det N={\bf 1}$
 if and only if $sym^2M$ contains a non trivial $F$ with $\partial F=0$.
\end{corollary}
\begin{proof} The `only if' part follows from the above observation. Now suppose that
$F$ exists. Then $F$ determines the symmetric bilinear form $<a,b>=F(a\otimes b)\in K$ on $M^*$. The subspace $\{a\in M^*|\ <a,M^*>=0\}$ is invariant under $\partial$ because $\partial F=0$. Since $M$ is irreducible, this subspace is 0 and so $F$ is non degenerate. Further, $F$ has a non trivial isotropic vector since $K$ is a $C_1$-field. \end{proof}

\begin{remarks}{\rm (1)  An extension of the test 2.2 is the following ($K$ is again a $C_1$-field).
Suppose that $M$ satisfies $\det M={\bf 1},\
M=sym ^2N$ for some module $N$ of dimension 2 with $L=\det N$. Suppose
that $L\neq {\bf 1}$. Then $L^{\otimes 3}= \det M={\bf 1}$ and
$\det (L\otimes N)={\bf 1}$ and $M':=sym ^2(L\otimes N)$ is equal to
$L^{\otimes 2}\otimes M$. Thus $sym ^2(M')$ contains an element $F'$
with the properties $\partial F'=0$, $F'$ non degenerate and has a non trivial isotropic vector. This translates for $M$ into the existence of a 1-dimensional submodule $KF=L\otimes KF'$ of $sym ^2M$ with $(KF)^{\otimes 3}={\bf 1}$. This submodule can be found, for instance, by computing $\ker (\partial ,{\rm End}(sym ^2M))$.  \\
(2) For an absolutely irreducible $M$ with $\det M={\bf 1}$ and $\dim M=3$, the above produces
an algorithm for the required $N$ with $sym ^2N=M$ .\\
 If $\dim M=n+1>3$, then an obvious test
whether $M\cong sym^nN$ holds for some  $N$ with $\dim N=2,\ \det N={\bf 1}$ (or, more generally, $M\cong (sym ^nN)\otimes L$ with $\dim L=1$)
is to perform the algorithm (explained in the beginning of this section) for producing $N$. If this does not fail, then one finally verifies whether $M$ is isomorphic to $sym^nN$ (or, more generally, to $(sym ^nN)\otimes L$).\\
(3) Theorem 2.1 is an example of the basic construction in Theorem 1.3. Indeed, suppose that $\det M={\bf 1}$ and that $(\frak{gal}(M),V(M))=(\frak{sl}_2,[2])$. Then one easily sees that ${\rm PSL}_2=Gal(M)^o\subset Gal(M)\subset \mu _3\cdot Gal(M)^o$. If $Gal(M)=Gal(M)^o$, then $M$ is an adjoint module and the proof of Theorem 2.1 is a special case of Theorem 1.3. \\
If $Gal(M)\neq Gal(M)^o$, then
$sym ^2M$ contains a unique 1-dimensional submodule with differential Galois group
$\mu _3=Gal(M)/Gal(M)^o$, i.e., the module $L$ from the above (1). One can replace
$K$ by the Picard-Vessiot field $K^+$ of $L$, which is a 3-cyclic extension and observes
that $K^+\otimes M$ is an adjoint module. One may also replace $M$ by the adjoint module $M'=L^{\otimes 2}\otimes M$.\\
(4) Another test would be to find relations between the solutions $V(M)$ of $M$. Suppose for convenience that $z=0$ is a non singular point for $M$. Then one can (approximately) compute $V(M)$ as
$\ker (\partial , \mathbb{C}((z))\otimes M)$ and it might be possible to read off relations. This geometric approach is present in Fano's paper  and is worked out by M.F.~Singer.  We present here a more detailed
version of the Fano-Singer theorem. }\hfill $\square$ \end{remarks}

In order to apply the geometry of projective varieties of small dimension, Fano has introduced the notion: \\
\noindent
{\it `The solutions of  a scalar equation of order $n$ lie on a variety $S\subset \mathbb{P}^{n-1}$'.} \\
 Let  $L\in K[\partial ]$ be monic of degree $n$. Its Picard-Vessiot ring has the form
 \[ K[Y_1,\dots ,Y_n,Y_1^{(1)},\dots ,Y_n^{(n-1)},\frac{1}{W}]/I\ ,\]
 where $W$ denotes the Wronskian and where $I$ a maximal differential ideal. We write
 $y_i$ for the image of $Y_i$ in this ring. Consider the homogeneous ideal $H$ in $\mathbb{C}[Y_1,\dots ,Y_n]$ generated by the homogeneous elements  $h\in \mathbb{C}[Y_1,\dots ,Y_n]$ that belong to $I$. (We note that this does not depend on the choice of the basis $Y_1,\dots ,Y_n$). It is clear that $H$ is a prime ideal.
 Let $S\subset \mathbb{P}(\mathbb{C}y_1+\cdots +\mathbb{C}y_n)=\mathbb{P}^{n-1}$
 be the variety defined by $H$. In Fano's terminology this is called: {\it `the solutions of $L$ lie
 on $S$'}. Since the $y_i$ are linearly independent, $S$ does not lie in a proper projective subspace of $\mathbb{P}^{n-1}$.

  \begin{theorem}[Fano--Singer] $\ $\\
 Suppose that the solutions of the monic operator $L$ of degree $n$  lie on a curve. Then one of the following holds: \\
{\rm (a)} After a shift $\partial \mapsto \partial +v$, all the solutions
of $L$ are algebraic.\\
 {\rm (b)} $L$ is the $(n-1)$th symmetric power of an order 2 operator $L_2$.\\
 {\rm (c)} There is a monic operator $L_2$ of degree 2, having a basis of solutions $w_1,w_2$, such that the polynomial $P:=(X-\frac{w_1'}{w_1})(X-\frac{w_2'}{w_2})$  lies in $K[X]$ and
 there exists an integer $N$  and a sequence $0=i_1<i_2<\cdots <i_n=N$ with g.c.d. 1, such that
 $S=\{w_1^{i_k}w_2^{N-i_k}|\ k=1,\dots ,n\}\subset
 \{w_1^s w_2^t|\ s+t=N\}$ is a basis of solutions for $L$. Moreover,
 $S$ is invariant under the permutation $w_1\leftrightarrow w_2$
 if  $P$ is irreducible. \end{theorem}

\begin{comments} {\rm
 Note that condition (a) does not at all imply that the solutions
 lie on a curve. Indeed, algebraic elements $y_1,\dots ,y_n$ over
 $K$ need not satisfy any non trivial  homogeneous relation
 over $\mathbb{C}$.
 In case (b), a basis of the solutions of $L$ can be written in the form
 $\{y_1^iy_2^{j}|\ 0\leq i,j;\ i+j=n\}$ where $y_1,y_2$ is a basis of the
 solutions of $L_2$. The curve is then the normal curve in
 $\mathbb{P}^{n-1}$.\\
 Case (c). Suppose that $P$ is reducible, i.e.,
 $v_i:=\frac{w_i'}{w_i}\in K$ for $i=1,2$. Let $S$ be a subset of
 $n$ elements of $\{w_1^sw_2^t|\ s+t=N\}$. Let $V$ denote the
 $\mathbb{C}$-vector space generated by $S$. This vector space is invariant
 under the differential Galois group and it   follows that there is a unique monic operator $L\in K[\partial ]$ with solution space $V$. Further, the  solutions of
 $sym ^NL_2$ lie on a normal curve in $\mathbb{P}^{N-1}$.
 The projection, given by the subset of the homogeneous coordinates, of this curve is a curve $\Gamma$ in $\mathbb{P}^{n-1}$. The solutions of $L$ lie on this curve.

Now suppose that $P$ is irreducible. Let $K^+=K(v_1,v_2)$ with
$v_i:=\frac{w_i'}{w_i}$. Let $S$ be a subset of $n$ elements of
$\{w_1^sw_2^t|\ s+t=N\}$, invariant under the permutation
$w_1\leftrightarrow w_2$. Let $W$ be the $\mathbb{C}$-vector space generated by $S$. Again, $V$ is invariant under the differential Galois group and, as above, there is a unique monic operator
$L\in K^+[\partial ]$ with solution space $W$. Since $S$ is invariant
under $w_1\leftrightarrow w_2$, the operator $L$ is invariant under
conjugation by the non trivial automorphism of $K^+/K$. Thus
$L\in K[\partial ]$. The solutions of $sym ^NL_2$ are still lying on a curve of genus 0 in $\mathbb{P}^{N-1}$ and its projection is again a curve in  $\mathbb{P}^{n-1}$.
We conclude that the converses of (b) and (c) hold.\\

 The condition on $L_2$, imposed in (c), can be translated into: The differential Galois group $G\subset {\rm GL}_2$ of $L_2$   is contained in  the group
  $\{{*\ 0\choose 0\ * }\}\cup \{{0\ *\choose *\ 0 }\}$. \\
  Let the polynomial  $X^2+cX+d=(X-v_1)(X-v_2)\in K[X]$ with discriminant
  $\Delta :=c^2-4d\neq 0$ be given. Then one can calculate that the corresponding  operator $L_2=\partial ^2+a\partial +b$ satisfies $a=c-\frac{\Delta '}{2\Delta}$ and $b=c+\frac{c'd-cd'}{\Delta}$. } \end{comments}

{\it The proof of the Fano--Singer theorem, revisited.}
\begin{proof}
Suppose that the solutions of $L$ lie on a curve $\Gamma \subset
\mathbb{P}(V)$, where $V=\mathbb{C}y_1+\cdots +\mathbb{C}y_n$ is the solution space. Then the differential Galois group $G\subset {\rm GL}(V)$
has image  $pG\subset {\rm PGL}(V)$ which acts as a group of
automorphisms of $\Gamma$. This action of $\Gamma$ is faithful
since $\Gamma$ does not lie in a proper projective subspace of
$\mathbb{P}(V)$. If the genus of the normalization of $\Gamma$
is $>1$ or if its genus is 1 and $\Gamma$ has a singular point, then $pG$ is finite. Then case (a) holds. If $\Gamma$ is non singular of
genus 1, then $(pG)^o$ is contained in $\Gamma$ seen as a group.
However the group $\Gamma$ is projective and $(pG)^o$ is affine.
Thus $(pG)^o=1$ and $pG$ is finite. This is again case (a).

Suppose that the normalization of $\Gamma$ has genus 0.
Let $\phi :\mathbb{P}^1\rightarrow \mathbb{P}^{n-1}$ be the morphism from the normalization of $\Gamma$ to $\Gamma$.  Then $\phi$ has
the form $(s:t)\mapsto (f_1:\cdots :f_n)$, where $f_1,\dots ,f_n$ are
homogeneous polynomials of the same degree in $s,t$ and the g.c.d.
of $f_1,\dots ,f_n$ is 1. Since $\Gamma$ is not contained in any
hypersurface, the $f_1,\dots ,f_n$ are linearly independent over $\mathbb{C}$.
As $\phi :\mathbb{P}^1\rightarrow \Gamma$ is birational, one has
$\mathbb{C}(\frac{t}{s})=\mathbb{C}(\frac{f_2}{f_1},\dots ,\frac{f_n}{f_1})$.
Let $d$ denote the degree of $\Gamma$. Then the $f_i$ have degree
$d$. Hence $n\leq d+1$.\\

{\it Consider the case $n=d+1$}. After a change of coordinates in
$\mathbb{P}^{n-1}$ we may suppose that $\phi$ has the form
$(s:t)\mapsto (s^{n-1},s^{n-2}t:\cdots :st^{n-2}:t^{n-1})$ and $\Gamma$
is the `rational normal curve'. Put $y_1=s^{n-1},\ y_2=s^{s-2}t,
\cdots,  \ y_n=t^{n-1}$. Then this basis of the solution space of $L$
has the relations $y_1y_3-y^2=0,\ y_2y_4-y_3^2=0,$ etc.. It follows that the differential Galois group $G$ of $L$ is a subgroup of
$Z_n\cdot sym ^{n-1}(GL_2(C))$, where $Z_n=\{\lambda \cdot 1|\ \lambda ^n=1\}$.

Write $PV\supset K$ for the Picard-Vessiot field of $L$ over
$K$. Consider the extension $PV^+=PV(u)$ of $PV$, defined by
the equation $u^{n-1}=y_1$. Define $v$ by $u^{n-2}v=y_2$. Then
$y_i=u^{n-i}v^{i-1}$ for all $i=1,\dots ,n$.  Let $G^+$ be the group of the differential automorphism  of $PV^+/K$.  Every element
$\sigma \in G$ has an extension $\sigma ^+$ to $G^+$. Indeed,
$\sigma (y_1)$ has the form $(au+bv)^{n-1}$ and the formula
$\sigma ^+(u)=\zeta (au+bv)$ (any $\zeta$ with $\zeta ^{n-1}=1$) produces $\sigma ^+$. Then one finds the following exact sequence

\[1\rightarrow Gal(PV^+/PV)\rightarrow G^+\rightarrow G
\rightarrow 1 \ .\]
Consider the vector space $W=\mathbb{C}u+\mathbb{C}v$ and the unique monic differential operator $L_2\in PV^+[\partial ]$ of degree 2 with
solution space $W$. Since $W$ is invariant under the action
of $G^+$, so is $L_2$. Further, one easily verifies that
 $(PV^+)^{G^+}=K$. Hence $L_2\in K[\partial ]$. Now clearly,
 $L$ is the $(n-1)$th symmetric power of $L_2$. \\

 \noindent
{\it Consider the case $d+1>n$}. The group $pG$ acts faithfully on
$\Gamma$ and this induces an action of $pG$ on the normalization
$\mathbb{P}^1$ of $\Gamma$. This embeds $pG$ into ${\rm PGL}_2$
with its usual action on $\mathbb{P}(\mathbb{C}s+\mathbb{C}t)$ and induced action on $\mathbb{P}(\mathbb{C}s^d+\mathbb{C}s^{d-1}t+\cdots +\mathbb{C}t^d)$. The projective subspace
$\mathbb{P}(\mathbb{C}f_1+\dots +\mathbb{C}f_n)$
 is invariant under $pG$
and under $(pG)^o$. The possibilities for $(pG)^o\neq 1$ are:
$\{ {*\ * \choose 0\ 1} \},\  \{  {1\ *\choose 0\ 1 } \},\  \{ {*\ 0\choose 0\ 1 } \}$.

In the first case $pG=(pG)^o$ and the invariant subspace under $pG$
are: $< \{s^at^b|\ a+b=d,\ b\leq i \}>$ for $i=0,1,\dots ,d$. If
 $<f_1,\dots ,f_n>$ is a proper subspace, then the g.c.d. of
 $f_1,\dots ,f_n$ is not 1, contradicting the form of $\phi$. The second
 case is excluded in the same way.

The third case has two subcases:
$pG=(pG)^o$ and $[pG:(pG)^o]=2$. \\
In the first subcase, the $pG$ invariant subspaces of dimension $n$
(with g.c.d. 1) have the form
$<\{y_i:=s^{d-a_i}t^{a_i}|\  0=a_1<a_2<\cdots <a_n=d \}>$. As before
one makes an extension $PV^+=PV(u,v)$ of the Picard-Vessiot
field $PV$ of $L$ over $K$ by equations
$u^{d-a_i}v^{a_i}=y_i$ for $i=1,\dots ,n$. The subspace
$\mathbb{C}u+\mathbb{C}v$
is invariant under the action of the group $G^+$ and yields the
required monic operator $L_2\in K[\partial ]$ of degree 2 with solution
space $\mathbb{C}u+\mathbb{C}v$. \\
In the second subcase, one shows that $pG=<(pG)^o, \
{0\ 1\choose 1\ 0}>$. This poses the extra condition on the
$pG$ invariant  subspaces, namely: for every $i$ there is a $j$ with $d-a_i=a_j$. Now one proceeds as in the first subcase. \end{proof}

\section{Differential modules with Lie algebra $\frak{sl}_3$}

According to Section 1, the {\it essential case} to consider is a differential module $N$ with $\det N={\bf 1}$ and $(\frak{gal}(N),V(N))=(\frak{sl}_3, [1,1])$.

 We  make the following observation.
If $N$ is the direct summand of ${\rm End}(M)$ for some $M$ with dimension
$3$,  then ${\rm SL}_3\subset Gal(M)\subset {\rm GL}_3$ and $Gal(M)$ acts as $PSL_3$
on $N$. Hence $N$ satisfies this property if and only if $Gal(N)$ is connected.

 In general, $Gal(N)$ is not connected and we will compute the minimal field extension $K^+\supset K$ such that $Gal(K^+\otimes _KN)$ is connected.

We identify $V(N)$ with $\frak{sl}_3$. The group $Gal(N)$ is contained in the normalizer
$G^+$ of $G^o:=Gal(N)^o={\rm PSL}_3$ in ${\rm SL}(\frak{sl}_3)$. The element
$\sigma \in {\rm GL}(\frak{sl}_3)$, defined by $\sigma (A)=-A^t$ for all $A\in \frak{sl}_3$,
has the property $\sigma G^o\sigma ^{-1}=G^o$ and $\sigma$ maps to the non trivial
element of $Out({\rm PSL}_3)$. The determinant of $\sigma $ is $-1$ and thus $\tau$, defined by $\tau =e^{2\pi i/16}\sigma$, lies  in $G^+$. Let $h$ be any element of $G^+$. Then, after multiplying $h$
by $\tau ^\epsilon$ with $\epsilon \in \{0,1\}$, we may suppose that the image of $h$ in
$Out({\rm PSL}_3)$ is 1. Thus there is an element $s\in {\rm PSL}_3$ with $hgh^{-1}=
sgs^{-1}$ for all $g\in G^o$. Since the $G^o$-representation $\frak{sl}_3$ is irreducible,
$s^{-1}h=\lambda \cdot 1$ with $\lambda ^8=1$. It follows that  $G^+$ is generated by
$\tau$ and $G^o$ and $[G^+:G^o]=16$. The possible groups $Gal(N)$
 satisfy $G^o\subset Gal(N)\subset G^+$ and are determined by the integer
 $d:=[Gal(N):Gal(N)^o]$ dividing 16.

The map $[\ ,\ ]:\Lambda ^2\frak{sl}_3 \rightarrow \frak{sl}_3$ is $G^o$-invariant and one verifies that $\tau ([\ ,\ ])=e^{2\pi i/16}\cdot [ \ ,\ ]$.  The $G^+$-module  ${\rm Hom}(\Lambda ^2\frak{sl}_3,\frak{sl}_3)$ has a unique 1-dimensional submodule, namely $\mathbb{C}[\ ,\ ]$. On this module $\tau$ acts as multiplication by
$e^{2\pi i/16}$.Then ${\rm Hom}(\Lambda ^2N,N)$ has a unique 1-dimensional submodule $L$. Now  $d$ is minimal such that $L^{\otimes d}={\bf 1}$. The $d$-cyclic extension $K^+\supset K$, defined by $L$, has the property that $K^+\otimes _KN$ is an adjoint module.  The algorithm of Theorem 1.3 computes the 3-dimensional differential module $M$ over
$K^+$ with differential Galois group ${\rm SL}_3$ such that ${\rm Hom}(M,M)=K^+\otimes N$.

\bigskip

\noindent  {\it An alternative method for the $\frak{sl}_3$ case}. \\
Consider the differential module $N\otimes N^*$. The corresponding $\frak{sl}_3$-module is $[3,0]\oplus [0,3]\oplus [2,2]\oplus [1,1]\oplus [1,1]\oplus [0,0]$.
The term $[3,0]\oplus [0,3]$ is invariant under the action of $G^+$, more precisely $\tau$ interchanges the terms $[3,0], [0,3]$ and $\tau ^2$ is the identity. Let $A$ be the submodule of $N\otimes N^*$ corresponding to $[3,0]\oplus [0,3]$. {\it If $\tau$ belongs to} $Gal(N)$, then one computes with the method of Subsection 1.1,  the quadratic extension
$K_2\supset K$ such that $K_2\otimes A$ splits  as a direct sum. Let $B$ be the direct summand corresponding to $(\frak{sl}_3,[3,0])$. If $\tau$ does not belong to $Gal(N)$, then we write $K_2=K$.

Now $Gal(B)={\rm PSL}_3$. Using the following constructions of linear algebra:  $\Lambda ^2[3,0]=[4,1]\oplus [0,3]$ and $[3,0]\otimes [0,3]=
[0,0]\oplus [1,1]\oplus [2,2]\oplus [3,3]$, one obtains a module $C$ corresponding to
$(\frak{sl}_3,[1,1])$ and $Gal(C)={\rm PSL}_3$. Thus $C$ is an adjoint module over $K_2$
and is induced by a standard module $M$ over $K_2$ for ${SL}_3$. The adjoint module
$\tilde{N}$ induced by $M$  is over the algebraic closure of $K_2$ isomorphic to the
input module $N$. Using Subsection 1.1 one finds the required extension $K^+$ of $K_2$. \\


\noindent {\it Test}. Let the input module $N$ be an absolutely irreducible  differential module with $\dim N=8,\ \det N={\bf 1}$. Then $(\frak{gal}(N),V(N))\cong (\frak{sl}_3,[1,1])$ if and only if  the irreducible direct summands of $sym^2N$ have dimensions $1,8,27$ and  the irreducible direct summands of $\Lambda ^2N $ have dimensions $8,10,10$, or $8,20$.\\

 \noindent {\it The above construction  generalizes
  to the case of a module $N$ where the representation
 $(\frak{g}(N),V(N)$ is the adjoint representation of $\frak{sl}_n$}. \\
 The element $\sigma$,
 as defined above, has determinant 1 for $n\equiv 1,2\mod 4$ and $-1$ for $n\equiv 0,3
 \mod 4$. In the first case, the normalizer $G^+$ of ${PSL}_n$ in ${\rm SL}(\frak{sl}_n)$ is generated by $\{\sigma, \ e^{2\pi i/(n^2-1)}\cdot 1\}$ and ${PSL}_n$. In the second case
 $G^+$ is generated by $\tau =e^{2\pi i/(2n^2-2)}\cdot \sigma$ and ${\rm PSL}_n$.
 In both cases $[G^+:{\rm PSL}_n]=2(n^2-1)$.

 {\it Surprisingly enough}, we have to distinguish the two cases $\det \sigma =1$ and
 $\det \sigma =-1$. In {\it the second case}, the group $G^+/{\rm PSL}_n$ acts faithfully
 on the 1 dimensional vector space $\mathbb{C}[\ ,\ ]$. The corresponding unique
 1-dimensional factor $L$ of ${\rm Hom}(\Lambda ^2N,N)$ yields, as before, the required
 extension $K^+$.

 In {\it the first case} the action of $G^+/{\rm PSL}_n$ on $\mathbb{C}[\ ,\ ]$ has kernel
 $\{1,\sigma\}$. If $\sigma$ lies in the image of $Gal(N)/Gal(N)^o$, then we first want
 to determine the quadratic extension $K_2\supset K$ which kills $\sigma$ (in other words, $\sigma$ does not lie in the image of $Gal(K_2\otimes N)/Gal(K_2\otimes N)^o$).
The alternative method for $\frak{sl}_3$ works here as well. Indeed, the module
$N\otimes N^*$ contains a unique irreducible direct summand $A$ which corresponds
to the $\frak{sl}_n$-module $[2,0,\dots ,0,1,0]\oplus [0,1,0,\dots ,0,2]$. We note that
$\sigma$ interchanges the two factors. The other irreducible summands of $N\otimes N^*$
correspond to irreducible $\frak{sl}_n$-modules which are invariant under $\sigma$.
The quadratic extension $K_2$ of $K$ for which $K_2\otimes A$ is a direct summand
of two irreducible submodules is the extension which kills $\sigma$. Now we replace
 $K$ by $K_2$, in case $\sigma$ lies in the image of $Gal(N)/Gal(N)^o$, and proceed as in the first case.

\section{Differential modules with Lie algebra $\frak{sl}_4$}
There are three connected linear algebraic groups with Lie algebra $\frak{sl}_4$,
namely ${\rm SL}_4,\ {\rm SL}_4/\mu _2={\rm SO}_6,\ {\rm SL}_4/\mu _4={\rm PSL}_4$. They correspond to
the  category $Repr_{\frak{sl}_4}=\{\{[1,0,0]\}\}$  and the two Tannakian
subcategories $\{\{[0,1,0]\}\}$ and $\{\{[1,0,1]\}\}$. The irreducible objects of the second category are the $[a,b,c]$ with $a-c+2b\equiv 0,2\mod 4$ and for the third category these objects are $[a,b,c]$ with $a-c+2b\equiv 0 \mod 4$. As noted before, the group
$Out({\rm SL_4})$ has two elements and  the non trivial element $\sigma$ in this group changes the representation $[a,b,c]$ into its dual $[c,b,a]$.
In Section 3 we treated the case of a module which induces the $\frak{sl}_4$-module
$[1,0,1]$ and this solves in principle the problem. For modules $P$ with
$\det P={\bf 1}$ such that $(\frak{gal}(P),V(P))$ is an irreducible object of $\{\{[0,1,0]\}\}$
not lying in $\{\{[1,0,1]\}\}$,  there is a construction of linear algebra producing a module
$M$ satisfying $\det M={\bf 1}$ and  $(\frak{gal}(M),V(M))=(\frak{sl}_4,[0,1,0])$. And
there is a construction of linear algebra from $M$ to $P$  (or maybe to a module $\tilde{P}$ which is isomorphic to $P$ over $\overline{K}$). \\
{\it Example}. Consider $P$ with $\det P={\bf 1}$ and
$(\frak{gal}(P),V(P)) =(\frak{sl}_4, [2,0,0])$. To obtain $M$, one makes for instance the steps:  $sym^2[2,0,0]=[4,0,0]\oplus [0,2,0]$; $[2,0,0]\otimes [0,2,0]$ has direct summand $[1,1,1]$; $[2,0,0]\otimes [1,1,1]$ contains the direct summand $[1,0,1]$; $[2,0,0]\otimes [1,0,1]$ contains the term $[0,1,0]$.\\
Now we describe a shortcut for modules of the above type $M$.\\

\noindent {\it Shortcut}. Let the differential module $M$ satisfy $\det M={\bf 1}$
and $(\frak{gal}(M),V(M))=(\frak{sl}_4,[0,1,0])$. We note that $[0,1,0]=\Lambda ^2[1,0,0]$
and thus $Gal(M)^o={\rm SL}_4/\mu _2$. The canonical morphism
$f: [0,1,0]\otimes [0,1,0]\rightarrow \Lambda ^4[1,0,0]=\mathbb{C}$ is a non degenerate
symmetric form. The group ${\rm SL}_4/\mu _2$ identifies with ${\rm SO}(f)\cong
{\rm SO}_6$. The normalizer of ${\rm SO}_6$ in ${\rm SL}_6$ is $\mu _3\cdot {\rm SO}_6$. Thus ${\rm SO}_6=Gal(M)^o\subset Gal(M)\subset \mu _3Gal(M)^o$. The module
$sym ^2M$ has a unique 1-dimensional submodule $L$ corresponding to
$\mathbb{C}f\subset sym^2[0,1,0]$. If $Gal(M)=Gal(M)^o$, then $L=KF$ where
$F$ is a non degenerate symmetric form on $N$ with $\partial F=0$. If $Gal(M)\neq
Gal(M)^o$, then $L$ determines a cyclic extension $K_3\supset K$ such that
$K_3\otimes _KL=K_3F$ where $F$ is a non degenerate symmetric form and
$\partial F=0$. The following theorem, which is an algorithm, finishes the description
of the shortcut.

 \begin{theorem} Let $M$ be a differential module of dimension 6.
 The following properties of $M$ are equivalent (no conditions on $M$ and $K$).\\
 {\rm (1)} $M\cong \Lambda ^2N$ for some module of dimension 4 with
 $\det N={\bf 1}$.\\
 {\rm (2)} There exists $F\in sym^2M$ with $\partial F=0$ such that
 $F$ is non degenerate and $M$ has a totally isotropic subspace of dimension 3.
 \end{theorem}
\begin{proof}
(1)$\Rightarrow$(2). Choose a basis $n_1,n_2,n_3,n_4$ of $N$ such that the corresponding matrix of $\partial$ has trace 0.  Then
 $\{n_{i,j}:=n_i\wedge n_j| \ 1\leq i<j\leq 4\}$ is a basis of
 $M=\Lambda ^2N$. The element $F=n_{12}n_{34}-n_{13}n_{24}+
n_{14}n_{23}$ is clearly non degenerate and has a totally isotropic subspace of dimension 3 over $K$. The operation $\partial$ on $N$
is given by a matrix $(\beta _{i,j})$, i.e., $\partial n_i=\sum _j\beta _{j,i}n_j$, such that
$\sum \beta _{i,i}=0$. A straightforward computation shows that $\partial F=0$.\\
\noindent (2)$\Rightarrow$(1). By assumption $F$ can be written in the form $m_1m_2+m_3m_4+m_5m_6$ for a suitable basis
$m_1,\dots ,m_6$ of $M$. For notational reason we write
$F=m_{12}m_{34}-m_{13}m_{24}+m_{14}m_{23}$ for a basis
$m_{12},\dots ,m_{34}$ of $M$. Let $(\alpha _{ij,kl})$ be the matrix of
$\partial$ on $M$ with respect to this basis, i.e., $\partial m_{ij}=
\sum \alpha _{kl,ij}m_{kl}$. The equality $\partial F=0$ is equivalent
to the set of equalities
\[\alpha _{ij,kl}=0 \mbox{ if } \{i,j,k,l\}=\{1,2,3,4\} \ ;\
 \alpha _{ij,ij}+\alpha _{kl,kl}=0\mbox{ for } ij\neq kl \ ;\]
\[\alpha _{ik,jk}=\pm \alpha _{ik',jk'}\mbox{ if }\{i,j,k,k'\}=\{1,2,3,4\}\mbox{ and} \]
the sign is $-$ for $\{i,j\}=\{1,3\},\ \{2,4\}$ and is $+$ for the other tuples
$\{i,j\}$. (Note that in the last set of relations we do not insist on
$i<k,\ j<k$ etc.).\\

We consider a vector space $N$ over $K$ with basis $n_1,\dots ,n_4$
and define the $K$-linear bijection $f:\Lambda ^2N\rightarrow M$ by
sending $n_{ij}:=n_i\wedge n_j$ to $m_{ij}$ for all $1\leq i<j\leq 4$.
On $N$ we consider an operation of $\partial$ given by a
matrix $(\beta _{i,j})$ (as above). The condition that $f$ is an isomorphism of differential modules is equivalent to a set of equations for the $\beta _{i,j}$. This set of equations can be computed to be
\[\beta _{i,i}+\beta _{j,j}=\alpha _{ij,ij}\mbox{ for } 1\leq i<j\leq 4\ ,  \
 \beta _{a,b}=\alpha _{aj,bj} \mbox{ if } a<b,\ a<j,\ b<j  \  ; \]
\[ \beta _{a,b}=-\alpha _{ai,ib}\mbox{ if } a<b,\ a<i,\ i<b\  ; \
 \beta _{a,b}=\alpha _{ia,ib}\mbox{ if } a<b, \ i<a,\ i<b \  ;\]
\[\beta _{a,b}=\alpha _{aj,bj}\mbox{ if } b<a,\ a<j,\ b<j\ ; \
\beta _{a,b}=-\alpha _{ja,bj}\mbox{ if } b<a,\ j<a,\ b<j\ ; \]
\[\beta _{a,b}=\alpha _{ia,ib}\mbox{ if } b<a,\ i<a,\ i<b \ .\]
This over-determined set of equations has a unique solution. Indeed,
one finds
\[\beta _{1,1}=(\alpha _{12,12}+\alpha _{13,13}-\alpha _{23,23})/2,\
\beta _{2,2}=(\alpha _{23,23}+\alpha _{24,24}-\alpha _{34,34})/2\ ,\]
\[\beta _{3,3}=(\alpha _{23,23}+\alpha _{34,34}-\alpha _{24,24})/2,\
\beta _{4,4}=(\alpha _{24,24}+\alpha _{34,34}-\alpha _{23,23})/2\ .\]
For each $a\neq b$ the above list gives two equations for
$\beta _{a,b}$. The two equations coincide, due to the relations
 $\alpha _{ik,jk}=\pm \alpha _{ik',jk'}$, listed above. \end{proof}

\begin{corollary} Let the differential field $K$ be a $C_1$-field and
let $M$ be an irreducible differential module of dimension 6. Suppose
that there exists an $F\in sym^2M$ with $F\neq 0$ and $\partial F=0$.
Then there exists an extension $K^+\supset K$ of degree $\leq 2$ and a differential module $N$ over $K^+$ such that $K^+\otimes _K M$ is isomorphic to $\Lambda ^2N$. Moreover $\det N={\bf 1}$.
\end{corollary}
\begin{proof} $F$ yields a symmetric bilinear form on $M^*$ defined by
$(a,b)=F(a\otimes b)$. The property $\partial (a,b)=(\partial a,b)+(a,\partial b)$
follows from $\partial F=0$. The $K$-vector space
$\{a\in M^*|\ (a,M^*)=0\}$ is invariant under $\partial$. Since $M^*$ is irreducible one
finds that this space is 0 and $F$ is non degenerate.

 Using that $K$ is a $C_1$-field one finds an expression
 $m_1m_2+m_3m_4+am_5^2+bm_6^2$ for $F$ with $a,b\in K^*$.
 If $b/a$ is a square in $K^*$, then $F$ has a totally isotropic subspace
 of dimension 3 and one can apply the theorem.

 In the other case, $F$ has a totally isotropic subspace of dimension 3 after tensorization with the field  $K^+:=K(\sqrt{b/a})$.    Finally,  the $\beta _{ii}$, found in the proof of the Theorem 4.1, satisfy  $\sum _{i=1}^4\beta _{i,i}=0$ and thus $\det (N)={\bf 1}$.   \end{proof}

\section{Differential modules with Lie algebra $\frak{sp}_4$}

The two groups associated to $\frak{sp}_4$ are ${\rm Sp}_4$ and
${\rm Sp}_4/\mu _2 ={\rm SO}_5$. Apart from operations of linear algebra in the category
$Repr_{\frak{sp}_4}$ we have to consider differential modules $M$ with $\det M={\bf 1}$ and $(\frak{gal}(M),V(M))=(\frak{sp}_4,[0,1])$.

\begin{proposition}[Test]  Let $M$ be absolutely irreducible of dimension 5 and
$\det M=1$.
The pair $(\frak{gal}(M), V(M))$ is isomorphic to $(\frak{sp}_4,[0,1])$ if and only if $sym ^2M$ is a direct sum of two irreducible spaces of dimensions $1,14$.\end{proposition}
\begin{proof}  (1) Suppose that $(\frak{gal}(M),V(M))=(\frak{sp}_4,[0,1])$, then $Gal(M)^o$ has Lie algebra $\frak{sp}_4$ and can only be $Sp_4$ or $SO_5=Sp_4/\{\pm 1\}$. Since $Sp_4$ has no faithful 5-dimensional irreducible representation one has $Gal(M)^o=SO_5$. Further  $SO_5\subset Gal(M)\subset (O_5\cdot \mathbb{C}^*)\cap
{\rm SL}(V(M))$. The latter group is  $\mu _5\cdot  \ {\rm SO}_5$.

Now $sym ^2V(M)$ splits as a direct sum of irreducible $Gal(M)$-modules of dimensions $1,14$.  Thus $sym^2M=L\oplus R$ with $L,R$ irreducible of dimensions $1,14$. If $Gal(M)=SO_5$, then $L={\bf 1}$. In the other case $L\neq {\bf 1}$  but  $L^{\otimes 5}={\bf 1}$. This $L$ determines a cyclic extension of $K$ of degree 5.\\
\noindent (2) Suppose that $sym^2M$ is the direct sum of irreducible modules of dimensions 1,14. The table of subsection 1.2 shows
that $(\frak{gal}(M),V(M))\cong (\frak{sp}_4,[0,1])$.
\end{proof}

{\it Algorithm}. Assume that $M$ {\it passes the test}. We want to produce a module $N$, of dimension 4 with $\det N={\bf 1}$ and $(\frak{gal}(N),V(N))=(\frak{sp}_4,[1,0])$, such that $M$ is a direct summand of $\Lambda ^2N$. \\

First we study the possibilities for $N$. One easily finds that $Gal(N)^o=Sp_4$ and
${\rm Sp}_4\subset Gal(N)\subset ({\rm Sp}_4\cdot \mathbb{C}^*)\cap {\rm SL}(V)$ and thus $Gal(N)$ is either $Sp_4$ or $Sp_4\cdot \mu _4$ (and then
$[Gal(N):Gal(N)^o]=2$).
Now $\Lambda ^2N=L\oplus R$, and $L$ generated by an
 element $F$ such that $\partial F=aF$ with $a=\frac{b'}{2b}$ for some $b\in K^*$. We try to find an isomorphism $R\rightarrow M$. The Galois group of $R$ equals $SO_5$ or
 $SO_5\cdot \mu _2$.  The latter is not contained in ${\rm SL}(V(M))$ and we conclude that  the $N$  that we want to produce has $Gal(N)=Sp_4$.\\

  If $Gal(M)\neq SO_5$,  then we have to replace $K$ by a cyclic extension $K^+$  of degree 5, in order to produce an isomorphism. The term $L$ in the proof of Proposition 5.1 has the form $L=Kb$ with $\partial b=\frac{f'}{5f}b$ for a suitable $f\in K^*$ and thus $K^+=K(\sqrt[5]{f})$ is the required extension.\\

 After replacing $K$ by $K^+$ (in case $L\neq {\bf 1}$), there is a $H\in sym ^2M$ with $H\neq 0,\ \partial H=0$. Since $M$
 is irreducible the form $H$ is non degenerate. As $K$ is a $C_1$-field, $H$ has an isotropic subspace of dimension 2. We consider now
 $M^+=M\oplus Ke$, with $\partial e=0$ and extend $H$ to a non degenerate symmetric form $H^+$ on $M^+$ with $\partial H^+=0$ and  such that $H^+$ has an isotropic subspace of dimension 3 (after possibly a quadratic extension of $K$).
 An application of Theorem 4.1 yields a module $N$ with $\det N={\bf 1}$ and  $\Lambda ^2N\cong M^+$.  From the form of $M^+$ one concludes that $N$ is the required module with $(\frak{gal}(N),V(N))= (\frak{sp}_4,[1,0])$.    \\

 We note that from an `input module' $P$ with  $\det P={\bf 1}$ and
  $(\frak{gal}(P),V(P))=(\frak{sp}_4, [2,0])$, one obtains a module $M$ with
  $\det M={\bf 1}$ and $(\frak{gal}(M),V(M)=(\frak{sp}_4,[0,1])$ by the step $sym^2[2,0]=[4,0]\oplus [0,2]\oplus [0,1]\oplus [0,0]$.

  \subsection{The cases $(\frak{sl}_5,[0,1,0,0])$ and $(\frak{so}_7,[0,0,1])$}
Both cases are `solved' by constructions of linear algebra. Indeed, one has  $sym^2[0,1,0,0]=[0,2,0,0]\oplus
[0,0,0,1]$ and the dual of $[0,0,0,1]$ is $[1,0,0,0]$. Further, $\Lambda ^2[0,0,1]=[0,1,0]\oplus [1,0,0]$.

\section{Some semi-simple  Lie algebras}

\subsection{$\frak{sl}_2\times \frak{sl}_2$}
${\rm SL}_2\times {\rm SL}_2$ is the simply connected group with
Lie algebra $\frak{sl}_2\times \frak{sl}_2$. The other connected groups with this Lie algebra are obtained by dividing ${\rm SL}_2\times {\rm SL}_2$ by a subgroup of its center $\mu _2\times \mu _2$. The category $Repr_{\frak{sl}_2\times \frak{sl}_2}$ has therefore four
full  Tannakian (proper) subcategories, namely the ones generated by one of the four modules $[1]_l\otimes [1]_r,\ [1]_l\otimes [2]_r,\ [2]_l\otimes [1]_r, \ [2]_l\otimes [2]_r$.
A generator for the category itself is
$([1]_l\otimes [0]_r)\oplus ([0]_l\otimes [1]_r)$.  We note that
$Out({\rm SL}_2\times {\rm SL}_2)=\mathbb{Z}/2\mathbb{Z}$. The non trivial element in this group is represented by $(A_1,A_2)\mapsto
(A_2,A_1)$. For an absolutely irreducible differential module
$M$ with $\det M={\bf 1}$ and $(\frak{gal}(M),V(M))$ equal to one the
above four cases (and other cases listed in Subsection 1.2), we will construct a differential module $N$ with $\det N={\bf 1}$ and $(\frak{gal}(N),V(N))=([1]_l\otimes [0]_r)\oplus ([0]_l\otimes [1]_r)$ such that (possibly after a field extension of
$K$) one has $M\in \{\{N\}\}$.

\subsubsection{$\frak{sl}_2\times \frak{sl}_2$ with $[1]_l\otimes [1]_r$}

{\it The problem}. $M$ is a given absolutely irreducible module of dimension 4 and $\det M={\bf 1}$. Test whether $M$ is, after a finite extension $K^+$ of $K$, equal to $N_1\otimes N_2$ with $\dim N_i=2$. Develop  an algorithm for computing $K^+,N_1,N_2$ in the positive case.

\begin{proposition} [Test] $\overline{M}:=\overline{K}\otimes _KM$ is isomorphic to a tensor product $N_1\otimes N_2$ of modules with dimension 2 if and only if  $sym^2M$ has a 1-dimensional factor $L$ such that  $L^{\otimes 4}={\bf 1}$.
\end{proposition}
\begin{proof} (1)  Suppose that $\overline{M}\cong N_1\otimes N_2$.  Without loss of generality we may assume that $\det M=\det N_1=\det N_2={\bf 1}$. Let $V$ be the solution space of $M$ and $V_1,V_2$ those of $N_1,N_2$. Then $V=V_1\otimes V_2$
and the differential Galois group $Gal(\overline{M})=Gal(M)^o$ of $N_1\otimes N_2$ is
an algebraic subgroup of ${\rm SL}(V_1)\otimes {\rm SL}(V_2)$. It is not a proper subgroup since $\overline{M}$ is irreducible.

The geometric interpretation, in the spirit of [F], of the tensor product $V=V_1\otimes V_2$ is the embedding $\mathbb{P}(V_1)\times \mathbb{P}(V_2)\rightarrow \mathbb{P}(V)$. The group $G^o:=Gal(M)^o={\rm SL}(V_1)\otimes {\rm SL}(V_2)$ preserves this
embedding. From this one deduces that  the normalizer $G^+ $ of $G^o $  in
${\rm SL}(V_1\otimes V_2)$ is generated by $G^o$ and an element $\tau$ which can be described as follows. Let $\sigma :V_1\rightarrow V_2$ be a linear bijection. Then
$\tau (v_1\otimes v_2)=e^{2\pi i/8}\cdot (\sigma ^{-1}v_2)\otimes (\sigma v_1)$. It is
 easy to verify that $\tau \in G^+$.  Obviously $\tau$ permutes the two factors of the tensor product.
 Any $g\in G^+$ preserves the set of pure
 tensors $\{v_1\otimes v_2|\ v_1\in V_1,\ v_2\in V_2\}$. Then $g$ can also permute
 the two factors of the tensor product or preserve them. After multiplication by $\tau$,
 if needed, we may suppose that $g$ has the form $A_1\otimes A_2$ with
 $A_i\in {\rm GL}(V_i)$ and we may write $g=\lambda \cdot (B_1\otimes B_2)$
 with $\lambda \in \mathbb{C}^*$ and $B_i\in {\rm SL}(V_i)$ for $i=1,2$. Then
 $\lambda ^4=1$ and since $\tau ^2=i$, one has that $\lambda =\tau ^{2j}$ for some
 $j$. This proves the statement concerning $G^+$. Further $[G^+:G^o]=4$ since
 $\tau ^4=-1\in G^o$.  The group $Gal(M)$ lies  between $G^o$ and $G^+$ and there
 are the following possibilities: $G=G^+=<G^o,\tau >,\ G=<G^o,\tau ^2>,\ G=G^o$.

Take a basis $e_1,e_2$ of $V_1$ and put $f_1=\sigma e_1,\ f_2=\sigma e_2$. The element $ h:=(e_1\otimes f_1)\otimes (e_2\otimes f_2)-(e_1\otimes f_2)\otimes (e_2\otimes f_2)$ in $sym^2(V_1\otimes V_2)$ is invariant under $G$. Moreover, $\mathbb{C}h$ is the unique line in $sym^2(V_1\otimes V_2)$, invariant under $G$. Further
$\tau h=i\cdot h$. There corresponds a unique 1-dimensional submodule $L\subset
sym^2M$. Let $d\geq 1$ be minimal with $L^{\otimes d}={\bf 1}$. The three possibilities
for $G$ correspond to $d=4,2,1$.

\noindent (2) Let the 1-dimensional $L$ be given. Then $L=Ke$ with $\partial e
=\frac{g'}{4g}e$ for some $g\in K^*$. After replacing $K$ by $K(\sqrt[4]{g})$, one has
that $L=KF$ with $\partial F=0$. The symmetric quadratic form $F$ is non degenerate
since $M$ is (absolutely) irreducible. After possibly a quadratic extension of $K$, the
form $F$ has an isotropic subspace of dimension 2. Now we apply Theorem 6.2.
\end{proof}

\begin{theorem} Let $M$ be a differential module over $K$ of dimension 4 with $\det M={\bf 1}$ (no further conditions on
$M$ and $K$).
Then $M$ is isomorphic to $A\otimes B$ for modules $A,B$ of dimension 2 and with $\det A=\det B={\bf 1}$ if and only if there exists
$F\in sym^2M$, $\partial F=0$, $F$ is non degenerate and has an isotropic subspace of dimension 2.
\end{theorem}
\begin{proof} Suppose that $M=A\otimes B$, with $\dim A=\dim B=2$ and
$\det A=\det B={\bf 1}$. There is a canonical isomorphism
$sym^2M\rightarrow ((sym ^2A)\otimes (sym^2B))\oplus ((\Lambda ^2A)\otimes (\Lambda ^2B))$. The second factor is, by assumption ${\bf 1}$
and is therefore generated by an element $F$ with $\partial F=0$. More explicitly, chose bases $a_1,a_2$ for $A$ and $b_1,b_2$ for $B$ such that the matrices for
$\partial$ on these bases have trace 0. Put $m_{ij}=a_i\otimes b_j$. Then the element
$F:=m_{11}\otimes m_{22}-m_{12}\otimes m_{21} \in sym ^2M$ satisfies $\partial F=0$. Further, the symmetric form $F$ is non degenerate and has an isotropic subspace of dimension 2.\\

Suppose now that $F\in sym ^2M$ with the required properties exists. Then for a suitable basis
$\{m_{ij} |\ 1\leq i,j\leq 2\}$ one has $F=m_{11}\otimes m_{22}-m_{12}\otimes m_{21}$. We consider now $K$-vector spaces
$A$ and $B$ with bases $a_1,a_2$ and $b_1,b_2$. Define a $K$-isomorphism $\phi :A\otimes B\rightarrow M$
by sending $a_i\otimes b_j$ to $m_{ij}$ for all $1\leq i,j\leq 2$. We make $A$ and $B$ into differential modules by putting
$\partial a_i=\sum \alpha _{j,i}a_j$ and $\partial b_i=\sum \beta _{j,i}b_j$. The two matrices $(\alpha _{i,j}),\ (\beta _{i,j})$
are as yet unknown. We only require that their traces are 0.

Let $\partial$ on $M$ be given by $\partial m_{ij}=\sum \gamma _{kl,ij}m_{kl}$. The assumption $\partial F=0$ leads
to the following equalities
\[\gamma _{11,22}=\gamma _{22,11}=\gamma _{12,21}=\gamma _{21,12}=0,\ \gamma _{11,11}+\gamma _{22,22}=\gamma _{12,12}+\gamma _{21,21}=0\ , \]
\[ \gamma _{12,11}=\gamma _{22,21},\  \gamma _{21,11}=\gamma _{22,12},\  \gamma _{12,22}=\gamma _{11,21},\
\gamma _{21,22}=\gamma _{11,12}\ .\]

The assumption that $\phi$ is an isomorphism of differential modules leads to a unique solution for the matrices
$(\alpha _{i,j}),\ (\beta _{i,j})$, namely
\[\alpha _{1,2}=\gamma _{11,21}=\gamma _{12,22},\  \alpha _{2,1}=\gamma _{22,12}=\gamma _{21,11}\ ,\]
\[ \alpha _{11}= (\gamma _{11,22}+\gamma _{12,12})/2,\ \alpha _{22}=(\gamma _{21,21}+\gamma _{22,22})/2\ , \]
 \[\beta _{1,2}=\gamma _{11,12}=\gamma _{21,22},\  \beta _{2,1}=\gamma _{22,21}=\gamma _{12,11}\ ,\]
 \[ \beta _{11}=(\gamma _{11,22}+\gamma _{21,21})/2,\ \beta _{22}=(\gamma _{12,12}+\gamma _{22,22})/2 \ .\]
 \end{proof}

\begin{remarks} {\rm $\ $\\
(1) Assume that $M$ satisfies $\det M={\bf 1}$ and $M=A\otimes B$ with
$\dim A=\dim B=2$ and {\it no condition on $\det A,\ \det B$}.

Write $A=L_1\otimes A_1,\ B=L_2\otimes B_1$ where
$L_1,L_2$ are modules of dimension 1 and $\det A_1=\det B_1={\bf 1}$. Put
$L=L_1\otimes L_2$, then $L^{\otimes 4}=\det M={\bf 1}$ and $sym ^2M=
L^{\otimes 2}\otimes sym^2(A_1\otimes B_1)$. The term $sym ^2(A_1\otimes B_1)$
contains $F_1$ as constructed  in the first part of the proof of Theorem 6.2. The corresponding factor $L^{\otimes 2}\otimes KF_1$ in $sym ^2M$ is a trivial module
if $L^{\otimes 2}={\bf 1}$. Otherwise, we have to determine in $sym ^2M$ a 1-dimensional submodule $R$ with $R^{\otimes 2}={\bf 1}$. Then one can either
replace $K$ by the quadratic extension defined by $R$ or replace $M$ by
$L\otimes M$, where $L$ is any 1-dimensional module with $L^{\otimes 2}=R$,
in order to be able to apply Theorem 6.2.\\
(2) In [F], p.496, one considers a differential {\it operator} $M$ of order 4 and assumes that  $M$ is not solvable by (possibly inhomogeneous) differential equations of order one and algebraic extensions. Moreover, it is assumed  that a basis of  solutions of $M$ satisfies a non trivial homogeneous equation over $\mathbb{C}$. The assertion is that there exists a differential operator  $\tilde{M}$, equivalent  to $M$, such that $\tilde{M}=M_1\otimes M_2$ with  $M_i$ of order 2. Moreover, for $\tilde{M}$ and also for the $M_i$ quadratic extensions of $K$ might be needed (see [vdP-S] for the notion of tensor product of operators and equivalence of operators).\\
(3) In [vH] an experimental algorithm for testing and solving  $L=L_1\otimes L_2$ (with $L$ of order 4 and $L_i$ of order 2) is given. Again a quadratic extension of $K$ might be needed. The possibilities of reducing order 4 differential operators to operators of smaller order are also studied  in [P]. } \hfill $\square$ \end{remarks}

\subsubsection{$\frak{sl}_2\times \frak{sl}_2$ with $[1]_l\otimes [2]_r$ (or $[2]_l\otimes [1]_r$)}
 Let $M$ be an absolutely irreducible differential module with $\det M={\bf 1}$ and
 $(\frak{gal}(M),V(M))=(\frak{sl}_2\times \frak{sl}_2,[1]_l\otimes [2]_r)$. The method of the later Subsection 6.2, combined with Section 2, yields the required decomposition $M=N_2\otimes N_3$.

\subsubsection{ $\frak{sl}_2\times \frak{sl}_2$ with $[2]_l\otimes [2]_r$}
 Let $M$ be an absolutely irreducible differential module with $\det M={\bf 1}$ and
 $(\frak{gal}(M),V(M))=(\frak{sl}_2\times \frak{sl}_2,[2]_l\otimes [2]_r)$.
 Write $V=V(M)=sym ^2V_1\otimes sym^2V_2$, where $V_1,V_2$ are the
 2-dimensional standard representations of the two components of $\frak{sl}_2\times \frak{sl}_2$. The group $G^o=Gal(M)^o$ equals ${\rm PSL}(V_1)\otimes {\rm PSL}(V_2)={\rm PSL}(V_1)\times {\rm PSL}(V_2)$. Let $\sigma :V_1\rightarrow V_2$ be a linear
 isomorphism. Define $\tau \in {\rm SL}(V)$ by
 $\tau :a\otimes b\mapsto e^{2\pi i/18}\cdot (sym^2(\sigma ^{-1})b)\otimes (sym^2(\sigma)a)$. The normalizer $G^+$ of $G^o$ in ${\rm SL}(V)$ is seen to be
 $<G^o,\tau >$ and $[G^+:G^o]=18$. Further $G^o\subset Gal(M)\subset G^+$.

 $\Lambda ^2M$ induces the $\frak{sl}_2\times \frak{sl}_2$-module $([2]_l\otimes [0]_r)\oplus ([0]_l\otimes [2]_r)\oplus ([2]_l\otimes [4]_r)\oplus ([4]_l\otimes [2]_r)$ with terms of dimensions $3,3,15,15$. If
 $\tau \in Gal(M)$, then $\Lambda ^2M$ has two irreducible components, dimensions $6,30$.
 Otherwise $\Lambda ^2M$ has four irreducible components. In both cases $\Lambda ^2M$
 contains a unique summand $A$ with solution space
 $([2]_l\otimes [0]_r)\oplus ([0]_l\otimes [2]_r)$ on which
 $Gal(M)$ acts faithfully. In particular, $M$ can be obtained from $A$ by constructions of linear algebra.
 Let $P_1,P_2$ denote the projections of  $([2]_l\otimes [0]_r)\oplus ([0]_l\otimes [2]_r)$ onto its components.
 One calculates that $\tau P_1=e^{2\pi i/9}P_2,\ \tau P_2=e^{2\pi i/9}P_1$ and thus
 $\tau (P_1-P_2)=-e^{2\pi i/9}(P_1-P_2)$.  Thus ${\rm End}(A)$ contains precisely
 two 1-dimensional submodules, a trivial one corresponding to $\mathbb{C}(P_1+P_2)$ and
 $L$ corresponding to $\mathbb{C}(P_1-P_2)$. Let $d>0$ be minimal with $L^d={\bf 1}$, then
 $d=[Gal(M):G^o]$ is a divisor of $18$. Write $L=Ke$ with $\partial e=\frac{g'}{dg}e$ for some $g\in
 K^*$ and define $K^+=K(\sqrt[d]{g})$. Then $K^+\otimes A$ has differential Galois group
 $G^o$ and decomposes as a direct sum of two differential modules of dimension 3 with
 differential Galois group ${\rm PSL}_2$. An application of Section 2 finishes this case.

 An {\it alternative method} is the following. The solution space of $M$ can be identified with $\frak{g}=\frak{sl}_2\times \frak{sl}_2$. The cyclic group $G^+/G^o$ acts faithfully on the space $\mathbb{C}[\ ,\ ]$,
 where $[\ ,\ ]:\Lambda ^2\frak{g}\rightarrow \frak{g}$ is the Lie operation. This induces a unique 1-dimensional submodule $L$ of ${\rm Hom}(\Lambda ^2M,M)$. Let $d>0$ be minimal with $L^{\otimes d}={\bf 1}$. Then $d=[Gal(M):Gal(M)^o]$ is a divisor of $18$ and $L$ defines a cyclic extension
 $K^+\supset K$ of dimension $d$. The module $K^+\otimes M$ is an adjoint module and an application
  of Theorem 1.3 finishes the algorithm.

 \subsubsection{$\frak{sl}_2\times \frak{sl}_2$ with $[1]_l\otimes [3]_r$ and $[1]_l\otimes [4]_r$}
 Let the absolutely irreducible differential module $M$ satisfy $\det M={\bf 1}$ and
 $(\frak{gal}(M),V(M))=(\frak{sl}_2\times \frak{sl}_2,[1]_l\otimes [3]_r)$. We want to reduce this module to the
 case $[1]_l\otimes [1]_r$.  One writes $V(M)=
 V_1\otimes sym^3(V_2)$ where $V_1,V_2$ are the standard representations of dimension 2
 of the two factors $\frak{sl}_2$. Then $Gal(M)^o={\rm SL}(V_1)\otimes sym ^3{\rm SL}(V_2)$.  The group $Gal(M)$ is contained in the normalizer of $Gal(M)^o$ in ${\rm SL}(V(M))$, which is $\mu _8\cdot Gal(M)^o$. In particular, every decomposition as $\frak{gal}(M)$-modules is also a decomposition as
 $Gal(M)$-module. Now $\Lambda ^2([1]_l\otimes [3]_r)$ contains the irreducible summand $[0]_l\otimes [2]_r$.
 Further $([0]_l\otimes [2]_r)\otimes ([1]_l\otimes [3]_r)$ contains the direct summand $[1]_l\otimes [1]_r$.

 We want to reduce the case $[1]_l\otimes [4]_r$ to $[1]_l\otimes [2]_r$. One has
 ${\rm SL}_2\otimes {\rm PSL}_2=Gal(M)^o\subset Gal(M)\subset \mu _{10}\cdot Gal(M)^o$ and every decomposition as $\frak{sl}_2\times \frak{sl}_2$-modules is also a decomposition as $Gal(M)$-module.
Now $sym^2([1]_l\otimes [4]_r)$ contains $[0]_l\otimes [2]_r$ and $([1]_l\otimes [4]_r)\otimes ([0]_l\otimes [2]_r)$
contains $[1]_l\otimes [2]_r$.

 \subsection{$\frak{sl}_2\times \frak{sl}_3$ with $[1]\otimes [1,0]$ and $[2]\otimes [1,0]$}
Let $M$ be an absolutely irreducible differential module such that $\det M={\bf 1}$ and $(\frak{gal}(M),V(M))=
(\frak{sl}_2\times \frak{sl}_3,[1]\otimes [1,0])$. The {\it problem} is to decompose $M$ as
$N_2\otimes N_3$ with $\dim N_i=i$, $\det N_i={\bf 1}$ and $(\frak{gal}(N_2),V(N_2))=
(\frak{sl}_2,[1])$ and $(\frak{gal}(N_3),V(N_3))=(\frak{sl}_3,[1,0])$.

The construction follows from the observation that $[1]\otimes [1,0]$ `generates' the Tannakian category $Repr_{\frak{sl}_2\times \frak{sl}_3}$. Explicitly,
$([1]\otimes [1,0])\otimes ([1]\otimes [1,0])$ is the direct sum of
$[0]\otimes [2,0],\ [0]\otimes [0,1],\    [2]\otimes [2,0],\  [2]\otimes [0,1]$ of dimensions
$6,3,18,9$. The corresponding direct sum decomposition of $\overline{M}\otimes \overline{M}$ is already present over $K$, since the Galois group $Gal(\overline{K}/K)$ cannot permute subspaces of distinct dimensions. Choose $N_3$ to be the dual of the factor of $M\otimes M$ of dimension 3.

Next, we consider $([1]\otimes [1,0])\otimes ([0]\otimes [0,1]) $ which is the direct sum of \\
$[1]\otimes [0,0],\ [1]\otimes [1,1]$ of dimensions $2,16$. As before, $M\otimes N_3^*$, has a direct summand of dimension 2 that we will call $N_2$. Then, by construction
 $\overline{M}$ and $\overline{N}_2\otimes \overline{N}_3$ are isomorphic. The solution space
 $V(M)$ has a decomposition as tensor product $V_2\otimes V_3$ of spaces of dimensions
 $2$ and $3$.  This is the unique decomposition, invariant under $Gal(M)^o$. Therefore this
 decomposition is also invariant under $Gal(M)$. Hence every element of $Gal(M)$ has the form
 $A_2\otimes A_3$ with $A_i\in {\rm GL}(V_i)$.  Combining this with the assumption that
 $Gal(M)\subset {\rm SL}(V(M))$ yields $Gal(M)=Gal(M)^o$. Thus the isomorphism between
 $\overline{M}$ and $\overline{N}_2\otimes \overline{N}_3$ is in fact an isomorphism between
 $M$ and $N_2\otimes  N_3$.\\

 For the case $[2]\otimes [1,0]$ one writes the solution space of a corresponding differential module $M$
 as $V(M)=(sym ^2V_1)\otimes V_2,\ \dim V_1=2,\ \dim V_2=3$. One verifies that $Gal(M)^o=(sym^2{\rm SL}(V_1))\otimes {\rm SL}(V_2)$ and that $Gal(M)$ is contained in $\mu _9\cdot Gal(M)^o$. The homomorphism $Gal(M)/Gal(M)^o\rightarrow Out(Gal(M)^o)$ is trivial since the latter group has order two.
Thus any direct sum decomposition of $\frak{gal}(M)$-modules is also a decomposition of $Gal(M)$-modules and of differential modules.  Now one makes the following  steps: the dual of $[2]\otimes [1,0]$
is $[2]\otimes [0,1]$; $([2]\otimes [1,0])\otimes ([2]\otimes [0,1])$ contains $[2]\otimes [0,0]$;
$([2]\otimes [0,0])\otimes ([2]\otimes [1,0])$ contains $[0]\otimes [1,0]$.

\subsection{$\frak{sl}_2\times \frak{sl}_4$ with $[1]\otimes [1,0,0]$ and similar cases}
Let the absolutely irreducible differential module $M$, with $\det M={\bf 1}$, induce the above case. Write
$V(M)=V_1\otimes V_2$ with $\dim V_1=2,\ \dim V_2=4$. The group $G^o=Gal(M)^o={\rm SL}(V_1)\otimes {\rm SL}(V_2)$ has the normalizer $G^+=\mu _8\cdot G^o$ and $[G^+:G^o]=2$.
The module $M\otimes M^*$ has a 3-dimensional summand $A_3$ corresponding to $[2]\otimes [0,0,0]$ and with differential Galois
group ${\rm PSL}(V_1)$. Using Section 2, one computes a module
$A_2$ with $\det A_2={\bf 1}$ and $sym^2A_2=A_3$.

Then $A_2\otimes M$ corresponds to $([0]\otimes [1,0,0])\oplus ([2]\otimes [1,0,0])$. The first term produces a direct summand $B_4$.
Then $M$ is isomorphic to $A_2\otimes B_4$ up to multiplication by
a 1-dimensional module $L$ with $L^{\otimes 2}={\bf 1}$.

The following cases can be `solved' in a similar way:
$\frak{sl}_2\times \frak{sp}_4$ with $[1]\otimes [1,0]$ and $[1]\otimes [0,1]$,
$\frak{sl}_2\times \frak{sl}_5$ with $[1]\otimes [1,0,0,0]$.

\subsection{$\frak{sl}_3\times \frak{sl_3}$ with $[1,0]_l\otimes [1,0]_r$}
This is a rather complicated case. Let $M$ be a differential module
with $\det M={\bf 1}$ and $(\frak{gal}(M),V(M))=(\frak{sl}_3\times \frak{sl_3},[1,0]_l\otimes [1,0]_r)$. The solution space
$V=V(M)$ has a decomposition $V_1\otimes V_2$ with $\dim V_1=\dim V_2=3$.
The group $G^o=Gal(M)^o={\rm SL}(V_1)\otimes {\rm SL}(V_2)$ acts in the obvious way.
The normalizer $G^+$ of $G^o$ in ${\rm SL}(V)$ can be seen to be $<G^o,\tau >$ where
$\tau$ is defined as follows. Choose a linear isomorphism $\sigma :V_1\rightarrow V_2$,
then $\tau :v_1\otimes v_2\mapsto e^{2\pi i/18}\cdot (\sigma ^{-1}v_2)\otimes (\sigma v_1)$.
Then $\tau ^2$ is multiplication by $e^{2\pi i/9}$ and $\tau ^6\in G^o$. Thus
$[G^+:G^o]=6$ and $G^o\subset Gal(M)\subset G^+$.

Consider $M\otimes M^*$. The tensor product $([1,0]_l\otimes [1,0]_r)\otimes ([0,1]_l\otimes [0,1]_r)$
has the irreducible components: $[1,1]_l\otimes [0,0]_r,\ [0,0]_l\otimes [1,1]_r,\ [0,0]_l\otimes [0,0]_r,\
[1,1]_l\otimes [1,1]_r$. If $\tau \not \in Gal(M)$, then all components belong to direct summands
of $M\otimes M^*$ and thus we find differential modules $M_1,M_2$ corresponding to the
first two terms.

If $\tau \in Gal(M)$, then $\tau$ interchanges the first two terms and $\tau ^2$ is the identity.
Thus $M\otimes M^*$ has an irreducible summand $A$ of dimension 16, which decomposes as a direct sum of two components after a quadratic extension of $K$.
This quadratic extension is found as described in Subsection 1.1 part (2).\\

 After, if needed, replacing $K$ by a quadratic extension, we may suppose that we
know the modules $M_1,M_2$ corresponding to $[1,1]\otimes [0,0]$ and
$[0,0]\otimes [1,1]$. The method of Section 3 produces, after possibly quadratic extensions of $K$, differential modules $N_1,N_2$ of dimension 3 and $\det N_i={\bf 1}$, such that  $M_i$ is the kernel of the obvious morphism $N_i\otimes N_i^*\rightarrow {\bf 1}$ for $i=1,2$.

 We note that $N_i$ is not unique, one may replace it by its dual  $N_i^{-1}$ and/or multiply it by  a 1-dimensional differential module $L$ with $L^{\otimes 3}={\bf 1}$.
 Thus, up to such a 1-dimensional differential module, there are four candidates for the tensor decomposition $M$,  namely  $N_1^{ \pm 1}\otimes N_2^{\pm 1}$. For a candidate $C$ one computes  whether the differential module ${\rm Hom}(C,M)$ has a 1-dimensional summand $L$  such that $L^{\otimes 3}={\bf 1}$. If such $L$ exists then $M\cong C\otimes L$. This solves the problem.  \\

 We observe that the finite extension of $K$, needed (together with two differential modules of dimension 3) to `solve' $M$ can be found by operations with
 $\Lambda ^2M$. The group $Gal(M)$ acts faithfully on the corresponding solution space
 $([2,0]_l\otimes [0,1]_r)\oplus ([0,1]_l\otimes [2,0]_r)$.
 The action of $\tau$ on $P_1,P_2$, the projections onto the two factors, can be seen to be
 $\tau P_1= e^{2\pi i/9}P_2,\ \tau P_2= e^{2\pi i/9}P_1$. Thus $\tau (P_1-P_2)=-e^{2\pi i/9}(P_1-P_2)$
 and there  corresponds a 1-dimensional submodule $L$ of ${\rm End}(\Lambda ^2M)$. Let $d>0$ be minimal with $L^{\otimes d}={\bf 1}$. Then $d$ divides $18$ and $L$ defines an explicit cyclic extension $K^+$ of $K$.

\subsection{$\frak{sl}_2\times \frak{sl}_2\times \frak{sl}_2$ with $[1]_l\otimes [1]_m\otimes [1]_r$}
Here $l,m,r$ stand for left, middle, right. This is again a rather complicated case. Consider a differential module $M$ with
$\det M={\bf 1}$ which induces the above representation of the Lie algebra of $Gal(M)$ on $V(M)$. One writes $V(M)=V_1\otimes V_2\otimes V_3$ and chooses identifications $a_i:\mathbb{C}^2\rightarrow V_i$
 for $i=1,2,3$. Then $Gal(M)^o=G^o={\rm SL}(V_1)\otimes {\rm SL}(V_2)\otimes {\rm SL}(V_3)$
 and the normalizer $G^+$ of $G^o$ in ${\rm SL}(V(M))$ has the form $(\mu _8\cdot G^o)\ltimes
 S_3$. We note that $[\mu _8\cdot G^o:G^o]=4$ and that the action of $S_3$ on $V(M)$ is
 given by: the permutation $\pi$ maps $v_1\otimes v_2\otimes v_3$ to
 $a_1a_{\pi (1)}^{-1}v_{\pi (1)}\otimes a_2a_{\pi (2)}^{-1}v_{\pi (2)}\otimes a_3a_{\pi (3)}^{-1} v_{\pi (3)}$.
One has $G^o\subset Gal(M)\subset G^+$. For convenience we write $[a;b;c]$ for the
$\frak{sl}_2\times \frak{sl}_2\times \frak{sl}_2$-module  with $[a]_l\otimes [b]_m\otimes [c]_r$.

The module $V(M)\otimes V(M)^*$ decomposes into the irreducible factors:\\
$[0;0;0],[2;0;0],[0;2;0],[0;0;2],[0;2;2],[2;0;2],[2;2;0],[2;2;2]$. The group $\mu _8 G^o$ acts here via its
quotient ${\rm PSL}(V_1)\otimes {\rm PSL}(V_2)\otimes {\rm PSL}(V_3)$ and $S_3$ permutes the
summands in the obvious way. We consider now the most complicated
case $Gal(M)=G^+$ and leave the other cases to the imagination. Then $M\otimes M^*$ has irreducible direct summands of dimensions $1,9,27,27$. The summand $A$ of dimension 9 has module
$[2;0;0]\oplus [0;2;0]\oplus [0;0;2]$. Then $\overline{A}=B_1\oplus B_2\oplus B_3$ where $B_1, B_2,B_3$ correspond to $[2;0;0],[0;2;0],[0;0;2]$.

Thus $\ker (\partial ,{\rm End}(\overline{A}))=\mathbb{C}P_1+
\mathbb{C}P_2+\mathbb{C}P_3$, where $P_i$ is the projection onto $B_i$ for $i=1,2,3$. It follows that
the differential module ${\rm End}(A)$ has a  summand corresponding to
$\mathbb{C}(P_1+P_2+P_3)$ and a 2-dimensional summand $T$ corresponding to the $S_3$-invariant subspace
$\{\lambda _1P_1+\lambda _2P_2+\lambda _3P_3|\ \lambda _i\in \mathbb{C},\ \sum _i \lambda _i=0\}$.
By construction, the differential Galois group of $T$ is the group $S_3$ and the Picard-Vessiot field
$K^+$ of $T$ is the field extension of $K$ needed for the decomposition of $M$ as a tensor product
with three factors. The Kovacic algorithm (slightly changed because $S_3\not \subset {\rm SL}_2$)
computes $K^+$. The terms $A_i$ of the decomposition $K^+\otimes _KA=A_1\oplus A_2\oplus A_3$
have differential Galois groups ${\rm PSL}(V_i)$. Using Section 2, one computes a module $M_i$
with $\det M_i={\bf 1}$ and $sym^2M_i\cong A_i$. Then $K^+\otimes _K M$ is isomorphic to
$(M_1\otimes M_2\otimes M_3)\otimes L^{-1}$ where $L$ is a suitable 1-dimensional module with
$L^{\otimes 4}={\bf 1}$. The term $L$ is the unique summand of ${\rm Hom}(K^+\otimes _K M,M_1\otimes M_2\otimes M_3)$ of dimension 1.\\

\noindent {\it Acknowledgments}. We thank Willem de Graaf and  Mark van Hoeij for their useful comments concerning this paper.\\

{\bf References}\\
\noindent [C-W] \'E.~Compoint and J.A.~Weil - {\it Absolute reducibility of differential operators and Galois groups} - Journal of Algebra 275 (2004) 77-105.\\
\noindent [F] G.~Fano - {\it \"Uber Lineare homogene Differentialgleichungen mit algebraischen Relationen zwischen den Fundamentall\"osungen } - Math.Ann.53 (1900), 493-590. \\
\noindent [vH] M.~van Hoeij -{\it Decomposing a 4'th order linear differential equation as a symmetric product} - Banach Center Publications, vol 58, p. 89-96, Institute of Mathematics, Polish Academy of Sciences, Warszawa 2002.\\
\noindent [vH-vdP] M.~van Hoeij and M.~van der Put - {\it  Descent for differential modules and skew fields} - J. Algebra 296 (2006), no 1, 18-55. \\
\noindent [H] J.E.~Humphreys - {\it Linear Algebraic Groups} - GTM 21, Springer-Verlag, 1987.\\
\noindent [J] N.~Jacobson - {\it Lie Algebras} - Dover Publications, Inc. New York, 1962.\\
\noindent [K] N.M~Katz - {\it Exponential sums and differential equations} - Annals of Mathematical studies, Number 124,  Princeton University Press, 1990.\\
\noindent [LiE] LiE online service - www-math.univ-poitiers.fr/~maavl/LiE/form.html \\
\noindent [N] K.A.~Nguyen - {\it On $d$-solvability for linear differential equations} -  preprint
November 2006.\\
\noindent [P] A.C.~Person - {\it Solving Homogeneous Linear Differential Equations of Order 4 in Terms of Equations of Smaller Order} - PhD thesis 2002, NCSU, \\
www.lib.ncsu.edu/theses/available/etd-08062002-104315\\
\noindent [vdP-S] M. van der Put  and M.F.~Singer - {\it Galois Theory of Linear Differential Equations} - Grundlehren 328, Springer-Verlag 2003.\\
\noindent [S1] M.F.~Singer -{\it Solving homogeneous linear differential equations in terms of second order linear differential equations} -
Am. J. Math. 107: 663-696, 1985.\\
\noindent [S2] M.F.~Singer - {\it Algebraic relations among solutions of linear differential equations} - Trans. Am. Math. Soc., 295: 753-763, 1986.\\
\noindent [S3] M.F.~Singer - {\it  Algebraic relations among solutions of linear differential equations: Fano's theorem} - Am. J. Math., 110: 115-144, 1988. \\

\end{document}